\documentclass[a4paper,10pt,notitlepage,reqno]{article}
\usepackage[english]{babel}
\usepackage[latin1]{inputenc}
\usepackage[T1]{fontenc}
\usepackage{amssymb,amsthm,amsmath} 
\usepackage{mathtools}
\usepackage{mathrsfs}
\usepackage{enumerate}
\usepackage{authblk}
\usepackage{stackrel}
\usepackage{eucal}
\usepackage{color}
\usepackage{xcolor}
\usepackage{graphics}
\usepackage{subcaption}
\usepackage{geometry}
\usepackage{hyperref}
\usepackage{float}
\usepackage[section]{placeins}
\numberwithin{equation}{section}
\newtheorem{Theorem}{Theorem}
\newtheorem{Corollary}{Corollary}
\newtheorem{Lemma}{Lemma}
\theoremstyle{definition}

\newtheorem{Remark}{Remark}

\newcommand{\eps}{\varepsilon}
\newcommand{\ups}{\upsilon}
\newcommand{\la}{\lambda}
\newcommand{\RR}{{\mathbb{R}}}
\newcommand{\CC}{{\mathbb{C}}}
\newcommand{\ZZ}{{\mathbb{Z}}}

\newcommand{\ii}{{\rm i}}
\newcommand{\cH}{\mathcal{H}}
\usepackage{color}

\newcommand{\comment}[1]{}
\DeclarePairedDelimiter{\abs}{\lvert}{\rvert}
\DeclarePairedDelimiter{\norm}{\lVert}{\rVert}

\DeclareMathOperator{\Realpart}{Re}
\renewcommand{\Re}{\Realpart}

\newcommand{\ie}{\emph{i.e.}}
\newcommand{\eg}{\emph{e.g.}}
\newcommand{\cf}{\emph{cf.}}
\newcommand{\etal}{\emph{et al.}}
\newcommand{\Com}{\mathbb{C}}

\newcommand{\Int}{\mathbb{Z}}

\newcommand{\dist}{\mathop{\mathrm{dist}}\nolimits}

\usepackage[normalem]{ulem}
\definecolor{DarkGreen}{rgb}{0,0.5,0.1} 
\newcommand{\txtD}{\textcolor{DarkGreen}}
\newcommand\soutD{\bgroup\markoverwith
{\textcolor{DarkGreen}{\rule[.5ex]{2pt}{1pt}}}\ULon}
\newcommand\soutP{\bgroup\markoverwith
{\textcolor{blue}{\rule[.5ex]{2pt}{1pt}}}\ULon}
\newcommand{\Hm}[1]{\leavevmode{\marginpar{\tiny%
$\hbox to 0mm{\hspace*{-0.5mm}$\leftarrow$\hss}%
\vcenter{\vrule depth 0.1mm height 0.1mm width \the\marginparwidth}%
\hbox to
0mm{\hss$\rightarrow$\hspace*{-0.5mm}}$\\\relax\raggedright #1}}}

\begin{document}
\title{\textbf{\LARGE 
Location of eigenvalues of non-self-adjoint discrete \\Dirac operators}}
\author{B.~Cassano,$^a$ \ O.~O.~Ibrogimov,$^b$
	\ D.~Krej\v{c}i\v{r}{\'\i}k,$^c$ \ and 
	\ F.~{\v S}tampach\,$^d$}
\date{\small 
\begin{quote}
\emph{
	\begin{itemize}
        \item[$a)$]
          Department of Mathematics,
          Universit\`{a} degli Studi di Bari,
          via Edoardo Orabona 4, 70125, Bari, Italy;
          biagio.cassano@uniba.it.
		\item[$b)$] 
		Institute for Theoretical Physics, ETH Z\"urich,
		Wolfgang-Pauli-Strasse~27, 8093 Z\"urich, Switzerland;
		oibrogimov@phys.ethz.ch.
		\item[$c)$] 
		Department of Mathematics, Faculty 
		of Nuclear Sciences and Physical 
		Engineering, Czech Technical University 
		in Prague,
		Trojanova 13, 12000 Prague~2, Czech 
		Republic; david.krejcirik@fjfi.cvut.cz.
		\item[$d)$] 
		Department of Applied Mathematics, 
		Faculty of Information Technology, 
		Czech Technical University in Prague, 
		Th{\' a}kurova~9, 16000 Praha, 
		Czech Republic; stampfra@fit.cvut.cz
	\end{itemize}
	}
\end{quote}
19 October 2019}
\maketitle
\begin{abstract}
\noindent
We provide quantitative estimates on the location of
eigenvalues of one-dimensional discrete 
Dirac operators with complex 
$\ell^p$-potentials for $1\leq p \leq\infty$. 
As a corollary, subsets of the essential spectrum free of embedded eigenvalues 
are determined for small $\ell^1$-potential. Further possible improvements and sharpness of the obtained spectral bounds are also discussed.
\end{abstract}
%
\section{Introduction}

\subsection{Motivation and state of the art}
The principal objective of this paper is to initiate 
a mathematically rigorous investigation of 
spectral properties of quantum systems 
characterised by a fusion of the following three features:
($\alpha$) relativistic, ($\beta$) discrete, ($\gamma$) non-self-adjoint.
While models within one of the respective classes have been intensively studied
over the last decades, the combination seems to represent a new challenging 
branch of mathematical physics. 

The relativistic feature~($\alpha$) 
is implemented by considering the \emph{Dirac} equation,
which is well understood in 
the simultaneously continuous and self-adjoint settings,
see~\cite{Thaller} for a classical reference.
Apart from describing relativistic quantum matter,
it also models quasi-particles in new materials like graphene.

The discrete feature~($\beta$) is due to introducing 
the Dirac operator on a \emph{lattice} rather than in the Euclidean space.
In the non-relativistic (Schr\"odinger) setting,
it is well known that the discretisation is not a mere shortcoming 
motivated by numerical solutions,
but it is in fact a more realistic model for semiconductor crystals,
see \cite{Boykin-Klimeck_2004}.
Indeed, it is essentially the tight-binding approximation in solid-state physics.
Works on the fusion ($\alpha$)$\cap$($\beta$) exist 
in the self-adjoint setting, 
see \cite{deOliv-Prad-JMP-2005,Car-deOliv-Prad-JMP-2011,
Gol-Hau-MFAT-2014,Kop-Tesch-JMAA-2016}
and references therein.

Finally, the non-self-adjointness~($\gamma$) is 
implemented through possibly \emph{non-Hermitian} perturbations
added to the free Dirac operator. 
Despite the new physical motivations 
coming from quasi-Hermitian quantum mechanics (\cf~\cite{Bagarello-book}),
there are very few results on non-self-adjoint 
Dirac operators in the literature.
For the fusion ($\alpha$)$\cap$($\gamma$) in the continuous setting,
see \cite{Cuenin-Laptev-Tretter_2014,Cuenin-IEOT14,Dubuisson_2014,
Sambou_2016,Enblom_2018,Cuenin-JFA-17,Cuenin-Siegl-LMP-18,Fan-Kre-LMP19}. 
For the complete combination ($\alpha$)$\cap$($\beta$)$\cap$($\gamma$),
we are only aware of the works \cite{Bair-Celeb-1999,Hulko-AMP-2019}
concerned with estimates on the number of discrete eigenvalues.
 
In this paper, we are interested in the location of eigenvalues
of \emph{one-dimensional} discrete Dirac operator perturbed 
by non-Hermitian potentials
(in particular, the coefficients of the potential are allowed to be complex).
The main ingredient in our proofs is the Birman--Schwinger principle
and the results are of the nature of the celebrated result 
of Davies \etal~\cite{Abramov-Aslanyan-Davies_2001} 
for one-dimensional continuous Schr\"odinger operators.	
However, following the strategy developed 
in~\cite{Fan-Kre-Veg-JST18} (see also \cite{Fan-Kre-Veg-JFA18} 
for an alternative approach), we manage to cover eigenvalues embedded 
in the essential spectrum as well. This article can be considered as 
a relativistic follow-up of~\cite{Ibr-Sta-19} by two of 
the present authors.

\subsection{Mathematical model}
Let $\{e_{n}\}_{n\in\ZZ}$ be the standard basis of the Hilbert 
space $\ell^{2}(\ZZ)$ and let $d:\ell^2(\ZZ) \to \ell^2(\ZZ)$ 
be the difference operator determined by the equation 
$
de_{n}:=e_{n}-e_{n+1}, n\in\ZZ.
$
The free discrete Dirac operator $D_0$ is a self-adjoint 
bounded operator in the Hilbert space $\ell^2(\ZZ)\oplus\ell^2(\ZZ)$ 
given by the block operator matrix
\begin{equation}\label{free.dirac}
	D_0 :=
	\begin{pmatrix}
	m & d \\[0.5ex]
	d^* & -m
	\end{pmatrix}, 
\end{equation}
where $m$ is a non-negative constant and $d^*$ is the adjoint 
operator to $d$ which fulfills $d^*e_{n}=e_{n}-e_{n-1}$,~$n\in\ZZ$.
It is well known that the spectrum of $D_0$ is absolutely 
continuous and is given by
\begin{equation}\label{spec.H0}
	\sigma(D_0)=\bigl[-\sqrt{m^2+4},-m\bigr]
	\cup
	\bigl[m,\sqrt{m^2+4}\bigr],
\end{equation}
see for instance~\cite{Gol-Hau-MFAT14}.

It is worth noting that $D_{0}$ can be represented by a doubly-infinite 
Jacobi matrix by using a suitably chosen orthonormal basis 
of $\ell^2(\ZZ)\oplus\ell^2(\ZZ)$. Indeed, if we set
\[
 f_{2n}:=0\oplus e_{n} \quad \mbox{ and } \quad f_{2n+1}:=e_{n}\oplus 0,
\]
for $n\in\ZZ$, then the matrix representation of $D_{0}$ with respect to 
the orthonormal basis $\{f_{n}\}_{n\in\ZZ}$ reads
\begin{equation}
 	D_0=\begin{pmatrix}
	\ddots & \ddots & \ddots &  &  &  & &      \\
	&  -1     & -m & 1     &        & & &     \\
	&  & 1     & m & -1     &         & &     \\
	&  & & -1     & -m & 1            & &     \\
	&  & & & 1     & m & -1             & \\
	&    &    &        &        & \ddots & \ddots & \ddots\\
	\end{pmatrix}.
\label{eq:D_0_per_Jac_mat}
\end{equation}

Moreover, it is often advantageous to view the above matrix as 
the $2\times2$-block tridiagonal Laurent matrix
\begin{equation}\label{free.dirac.matrix}
	D_0=\begin{pmatrix}
	\ddots & \ddots & \ddots &        &        & &       \\
	&  a^T     & b & a      &        & &       \\
	&        & a^T      & b  & a      & &     \\
	&        &        &  a^T     & b  & a &     \\       &        &        &        & \ddots & \ddots & \ddots\\
	\end{pmatrix},
\end{equation}
where 
\begin{equation}
b:=\begin{pmatrix}
-m & 1\\
1 & m
\end{pmatrix}, 
\quad 
a:=\begin{pmatrix}
0 & 0\\
-1 & 0
\end{pmatrix}. 
\end{equation}
The matrix~\eqref{free.dirac.matrix} naturally determines 
a unique operator acting on the Hilbert space 
$\cH:=\ell^{2}(\ZZ,\CC^{2})$ that is unitarily equivalent 
to~$D_{0}$ given by~\eqref{free.dirac}. 
We do not distinguish 
the unitarily equivalent operators
in the notation.

Further, we intend to perturb~\eqref{free.dirac.matrix} by the 
$2\times2$-block diagonal matrix 
\begin{equation}\label{V}
 V=\bigoplus_{n\in\ZZ}\ups_{n},
\end{equation}
where
\begin{equation}\label{mat.ups_n}
	\ups_n:=
		\begin{pmatrix}
		\ups^{11}_n & \ups^{12}_n
		\\[1ex]
		\ups^{21}_n & \ups^{22}_n
		\end{pmatrix}
\end{equation}
is a given sequence of $2\times 2$ complex matrices. We denote the 
resulting operator by $D_{V}:=D_{0}+V$. In view of the initial 
setting~\eqref{free.dirac}, such perturbation corresponds to a 
perturbation of each of the four operator entries by a diagonal 
matrix operator acting on~$\ell^{2}(\ZZ)$. For special symmetric choices of the coefficients, the perturbations of the diagonal entries represent an electric potential 
while the off-diagonal elements introduce a magnetic potential to $D_{0}$;
we proceed in a greater generality by making no hypotheses about the coefficients 
except for summability conditions.

In this paper we are concerned with the location of eigenvalues 
of the Dirac operator~$D_V$. If the entries of $\ups_{n}$ vanish as 
$n\to\pm\infty$, \ie~$\ups^{ij}_n\to0$ as $n\to\pm\infty$, then $V$ is compact and hence the essential spectrum of the perturbed operator~$D_V$ coincides with $\sigma(D_0)$.
The goal of the present paper is to investigate the location 
of the point spectrum of~$D_{V}$. To this end, we consider the 
block diagonal matrix potentials $V$ given by \eqref{V}\txtD{--}\eqref{mat.ups_n} 
which belong to the Banach space
$\ell^p\bigl(\ZZ,\CC^{2\times2}\bigr)$ equipped with the norm
\begin{equation}\label{potent.norms}
	\|V\|_p=\bigg(\sum_{n\in\ZZ} |\ups_n|^p\bigg)^{1/p},
	\quad 1\leq p<\infty, 
	\qquad \|V\|_{\infty}=\sup_{n\in\ZZ}|\ups_n|,
\end{equation}
where $|\ups_n|$ denotes the operator norm of the matrix 
$\ups_n$. As it is seen by comparing~\eqref{V} and~\eqref{potent.norms}, 
we slightly abuse the notation by not distinguishing between $V$ as the operator
and $V$ as the doubly-infinite $2\times2$-matrix valued sequence, 
whenever suitable. Except the notation $|\ups|$ used for the spectral norm of a matrix $\ups \in \Com^{2 \times 2}$, we denote by $|\ups|_\mathrm{HS}$ the Hilbert-Schmidt (or Frobenius) norm of~$\ups$ throughout the paper. Recall that $|\ups| \leq |\ups|_\mathrm{HS}$.

\subsection{Main results}
Our main result for $\ell^1$-potentials reads as follows:
\begin{Theorem}\label{thm:1-norm}
Let $V\in\ell^1\bigl(\ZZ,\CC^{2\times2}\bigr)$. Then 
\begin{equation}\label{sp.enc.l1}
	\sigma_{\rm{p}}(D_V)\subset
	\Big\{\la\in\CC \;\big|\;\; 
	|\la^2-m^2||\la^2-m^2-4| \leq
	\bigl(|\la+m|+|\la-m|\bigr)^2\bigl\|V\bigr\|^2_1
	\Big\}.
\end{equation}
\end{Theorem}
The spectral enclosure in~\eqref{sp.enc.l1} is a compact set symmetric with respect to both the real and the imaginary line. The geometry of its boundary is quite easy to understand. It is an algebraic curve of generically three possible topological configurations depending on the $\ell^{1}$-norm of the potential $V$ and the parameter $m>0$. A closer inspection of the respective polynomial equation shows that, if
\begin{equation}\label{Q0}
	0<\|V\|_{1}^{2}<\frac{m^2}{2}+1-m\sqrt{\frac{m^2}{4}+1},
\end{equation}
the boundary curve consists of four simple closed curves having the end-points of the essential spectrum $\pm m$ and $\pm\sqrt{m^{2}+4}$ in their interiors, respectively. If 
\[
 \frac{m^2}{2}+1-m\sqrt{\frac{m^2}{4}+1}<\|V\|_{1}^{2}<\frac{m^{2}}{4}+1,
\]
the boundary curve comprises two simple closed curves with the intervals $[-\sqrt{m^{2}+4},-m]$ and $[\sqrt{m^{2}+4},m]$ in their interiors, respectively. Finally, for
\[
 \|V\|_{1}^{2}>\frac{m^{2}}{4}+1,
\]
the boundary curve is a closed simple curve with the interval $[-\sqrt{m^{2}+4},\sqrt{m^{2}+4}]$ in its interior. Figure~\ref{fig:ell1plots} shows all the topological configurations.  
\begin{figure}
	\centering
 	\includegraphics[width=0.85\textwidth]{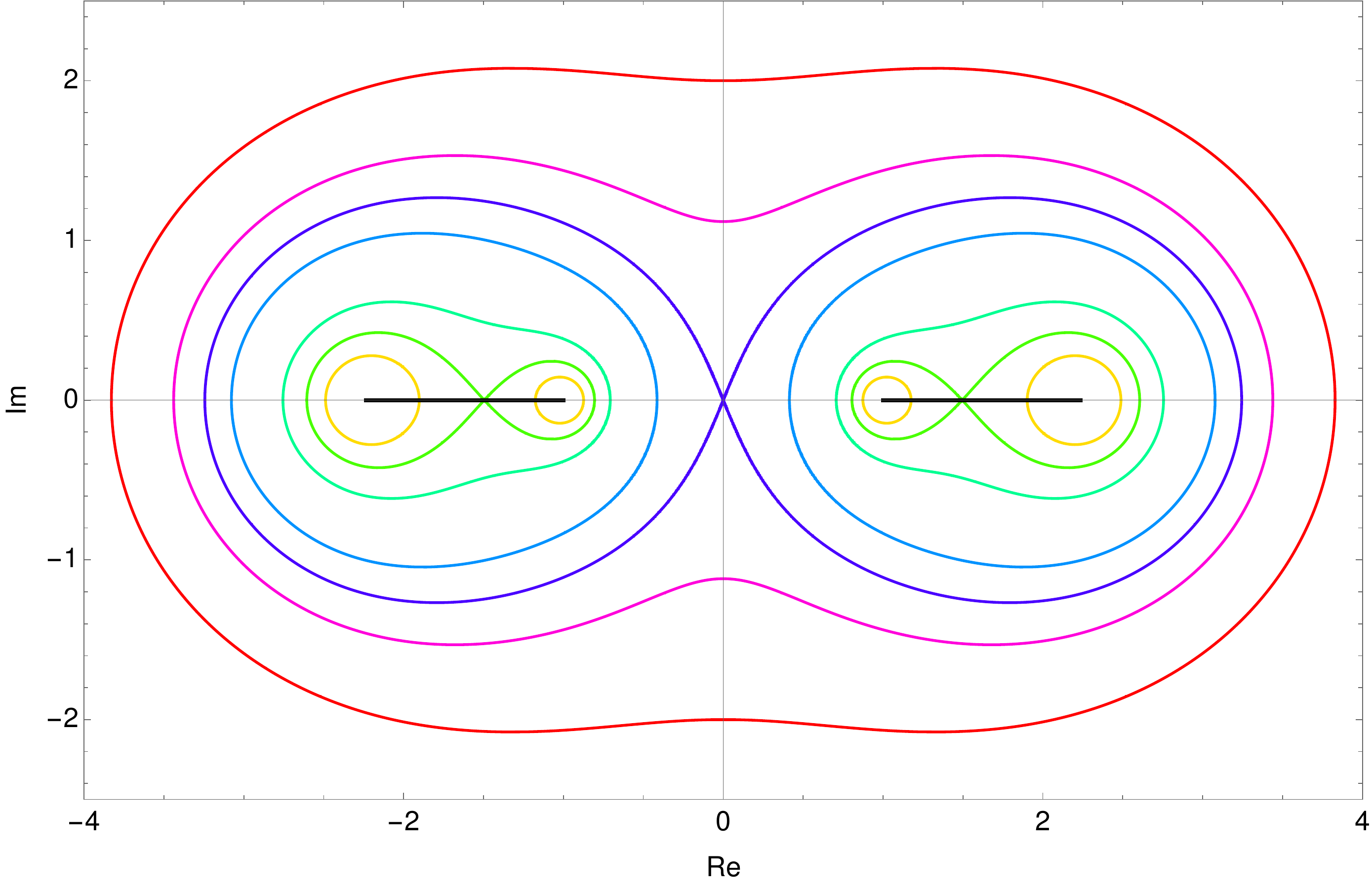}
 	\caption{\small{The plots of the expanding boundary curves 
 			corresponding to the spectral enclosure~\eqref{sp.enc.l1} 
 			for various values of 
 			$\|V\|_1\in[0.5, 1.5]$
 			 and $m=1$. The black intervals indicate the essential spectrum of $D_V$.}}
    \label{fig:ell1plots}
\end{figure}

As an immediate corollary of the firstly mentioned possible configuration for the boundary curve of~\eqref{sp.enc.l1}, we obtain subsets of the essential spectrum of~$D_V$ that are free of embedded eigenvalues of~$D_V$. 

\begin{Corollary}\label{cor:emb.EV.free.region}
	If the potential $V\in\ell^1\bigl(\ZZ,\CC^{2\times2}\bigr)$ satisfy~\eqref{Q0}, then the union of intervals
	\[
	(-\lambda_{+},-\lambda_{-})\cup(\lambda_{-},\lambda_{+}),
	\]
	where
	\begin{equation}\label{la.pm}
	\lambda_{\pm}:=\sqrt{m^{2}+2
		\left(1-\|V\|_{1}^2 \pm\sqrt{\left(1-\|V\|_{1}^2\right)^2-\|V\|_{1}^2m^2}\,\right)},
	\end{equation}
	is free of embedded eigenvalues of $D_V$.
\end{Corollary}
\begin{Remark}
 In fact, Corollary~\ref{cor:emb.EV.free.region} can be improved under additional assumptions that $1+\ups_{n}^{12}\neq0$ and $1+\ups_{n}^{21}\neq0$ for all $n\in\ZZ$. In this case, the whole interior of~\eqref{spec.H0} is free of embedded eigenvalues of~$D_{V}$. Indeed, let $D_{V}$ be viewed as the 2-periodic Jacobi matrix~\eqref{eq:D_0_per_Jac_mat} correspondingly perturbed by~$V$ for the moment. Then, if $\ups^{ij}\in\ell^{1}(\ZZ)$ for all $i,j\in\{1,2\}$, there exist two linearly independent solutions $\phi^{(\pm)}$ of the eigenvalue equation $D_{V}\phi=\lambda\phi$ such that $\phi_{n}^{(\pm)}\sim w^{\pm n}a_{n}$, as $n\to\infty$, for a nontrivial $2$-periodic sequence $a_{n}$, where $|w|=1$ provided that $\lambda\in(-\sqrt{m^2+4},-m)\cup(m,\sqrt{m^2+4})$. 
 As a result, there cannot be a square summable solution of $D_{V}\phi=\lambda\phi$, for $\lambda\in(-\sqrt{m^2+4},-m)\cup(m,\sqrt{m^2+4})$, if the set of solutions is of dimension 2 which is guaranteed by the additional assumptions $1+\ups_{n}^{12}\neq0$ and $1+\ups_{n}^{21}\neq0$ for all $n\in\ZZ$.
 This was proved in~\cite{Ger-Van-JAT86} for a certain real $\ell^{1}$-perturbations $V$. The reality is, however, inessential for the proof and the claim can be extended to complex $\ell^{1}$-perturbations as well, see the proof of~\cite[Thm.~3]{Ger-Van-JAT86}.
\end{Remark}

Our next result provides a spectral estimate in terms of 
the $\ell^p$--norm of the potential for $p>1$. 
The strategy of its derivation relies on an application 
of Stein's complex interpolation theorem to an appropriate 
analytic family of Birman-Schwinger type operators. 
This approach was successfully used in the continuous 
setting recently, see e.g.~\cite{Cuenin-JFA-17,Frank-TAMS-18, Coss2019}. 

\begin{Theorem}\label{thm:p-norm.stein}
	Let $1<p\leq\infty$ and $V\in\ell^p\bigl(\ZZ,\CC^{2\times2}\bigr)$. 
	If $\la\notin\sigma(D_0)$ satisfies 
	\begin{equation}\label{eq:p-norm.2}
	g_p(\la,m) \|V\|_p<1
	\end{equation}
	with
	\begin{equation}\label{func.g}
	g_p(\la,m):=\frac{(|\la-m|+|\la+m|)^{\frac{1}{p}}}{|\lambda^2-m^2|^{\frac{1}{2p}}{|\lambda^2 -m^2-4|}^{\frac{1}{2p}}(\dist(\la, \sigma(D_0)))^{1-\frac{1}{p}}},
	\end{equation}
	then $\la\notin\sigma(D_V)$.
\end{Theorem}
\begin{Remark}
For $p=\infty$, the function in \eqref{func.g} has to be understood as 
\[
g_\infty(\la,m)=\frac{1}{\dist(\la,\sigma(D_0))}.
\]
\end{Remark}

Third theorem concerns with spectral bounds for 
$\ell^{p}$-potentials with $p\geq1$ again.
In particular, the bound for $\ell^{1}$-potentials is an improvement of Theorem~\ref{thm:1-norm}.
The price we have to pay, however, is that the new bounds are quite complicated and not entirely explicit since they involve spectral norms of the matrices
\begin{equation}
 T_{0}(k)\!=\!\frac{1}{k^{-1}-k}\begin{pmatrix}
                                \lambda-m & 1-k \\
                                1-k & \lambda+m
                            \end{pmatrix}
 \;\,\mbox{ and }\;\,
 T_{1}(k)\!=\!\frac{k}{k^{-1}-k}\begin{pmatrix}
                                \lambda-m & 1-k \\
                                1-k^{-1} & \lambda+m
                            \end{pmatrix}\!,
\label{eq:def_T0_T1}
\end{equation}
which arise in the formula for the resolvent operator $(D_{0}-\lambda)^{-1}$, see Section~\ref{Sec.free}. Moreover, in contrast to Theorems~\ref{thm:1-norm} and~\ref{thm:p-norm.stein}, the spectral enclosures are not expressible entirely in the spectral parameter~$\lambda$. Rather than that they use the auxiliary parameter~$k$ with $|k|<1$. The relation between $\lambda$ and $k$ is determined by the equality
\begin{equation}\label{Jouk}
 \la^{2}=m^2+2-k-k^{-1}
\end{equation}
which introduces a one-to-two mapping $\lambda=\lambda(k)$ between the punctured unit disc $\{k\in\CC \mid 0<|k|<1\}$ and the resolvent set~$\rho(D_{0})$. This mapping plays the same role as the Joukowski mapping in the case of discrete Schr{\" o}dinger operator, see~\cite{Ibr-Sta-19} for details. 

The proof of the following theorem is based on a discrete version of Young's inequality.
Here and in the sequel, for $p\in(1,\infty]$, we denote by 
$q\in[1,\infty)$ the corresponding H\"older exponent, 
\ie~$q=p/(p-1)$ if $1<p<\infty$ and $q=1$ if $p=\infty$.

\begin{Theorem}\label{thm:spec_bound_improved}
Let $1 \leq p\leq\infty$ and assume 
$V\in\ell^p\bigl(\ZZ,\CC^{2\times2}\bigr)$. 
If $\la\notin\sigma(D_0)$ satisfies 
\begin{equation}\label{eq:p-norm_impr}
h_q(\la,m) \|V\|_p<1
\end{equation}
then $\la\notin\sigma(D_V)$,
with
\begin{equation}
h_q(\la,m):=
\begin{cases}
\left(|T_{0}(k)|^{q}+\frac{2}{1-|k|^{q}}|T_{1}(k)|^{q}\right)^{\!1/q}
\quad &\text{ if }1\leq q < \infty,
\\
\max\left\{|T_{0}(k)|,|T_{1}(k)|\right\} \quad &\text{ if }q=\infty, 
\end{cases}
\end{equation}
and $k$ the unique point in the punctured unit disk
$\{k\in\CC\;|\;\; 0<|k|<1\}$ such that $\la^{2}=m^2+2-k-k^{-1}$.
The matrices $T_{0}(k)$ and $T_{1}(k)$ are defined
in~\eqref{eq:def_T0_T1}.
\end{Theorem}
\begin{Remark}
Clearly, spectral norms of the $2\times2$ matrices~\ref{eq:def_T0_T1} can be expressed explicitly but the resulting formulas are somewhat cumbersome. Namely, we have

\begin{equation}
|T_{1}(k)|^{2}
=\abs{k}^{2}
\frac{\abs{\lambda+ m}^2 + \abs{\lambda-m}^2 +
	(\abs{k}+ \abs{k}^{-1})\abs{\lambda^2 - m^2}}%
{{\abs{\lambda^2 - m^2}\abs{\lambda^2 - m^2-4}}}
\label{eq:T1_norm}
\end{equation}
and
\begin{equation}
|T_{0}(k)|^{2}=\frac{1}{2}\left(B+\sqrt{B^2 - 4C}\right),
\label{eq:T0_norm}
\end{equation}
where
\begin{equation*}
B = 
\frac{\abs{\lambda+m}^2 + \abs{\lambda -m}^2 +
	2\abs{k}\abs{\lambda^2-m^2}}%
{\abs{\lambda^2 - m^2}\abs{\lambda^2 - m^2-4}}
\quad\mbox{ and }\quad
C = 
\frac{\abs{k}}{\abs{\lambda^2-m^2 -4}}.
\end{equation*}
\end{Remark}
\begin{Remark}
	We do not discuss the eigenvalues possibly embedded in~$\sigma(D_0)$ in Theorem~\ref{thm:spec_bound_improved} for $p=1$ similarly as is done in Corollary~\ref{cor:emb.EV.free.region} after Theorem~\ref{thm:1-norm}. Nevertheless, an inspection of the intersection points of the boundary curve of the spectral enclosure of Theorem~\ref{thm:spec_bound_improved} 
	with $\sigma(D_0)$ (when they exist) shows that they actually coincide with the points identified in Corollary~\ref{cor:emb.EV.free.region}. Indeed, it readily follows from formulas~\eqref{eq:T0_norm} and~\eqref{eq:T1_norm} that
	\[
	\lim_{k\to k_{0}}|T_{0}(k)|\leq
	\frac{\abs{\lambda_0 +m} + \abs{\lambda_0 -m}}
	{\sqrt{\abs{\lambda_0^2 - m^2}\abs{\lambda_0^2 -m^2 -4}}}=\lim_{k\to k_{0}}|T_{1}(k)|,
	\]
	for $\lambda_0 \in\sigma(D_0) \setminus \{\pm m, \pm \sqrt{m^2 +4}\}$ and $k_{0}$ a point on the unit circle such that $\lambda_{0}^{2}=m^{2}+2-k_{0}-k_{0}^{-1}$. Consequently, 
	\[
	\lim_{\lambda \to \lambda_0} h_{\infty}(\lambda,m) =
	\frac{\abs{\lambda_0 +m} + \abs{\lambda_0 -m}}%
	{\sqrt{\abs{\lambda_0^2 - m^2}\abs{\lambda_0^2 -m^2 -4}}},
	\]
	which is the expression appearing in the spectral enclosure of Theorem~\ref{thm:1-norm}. Consequently, even if the intervals of~$\sigma(D_0)$ given by the intersection points of the boundary curves of the improved spectral enclosure from Theorem~\ref{thm:spec_bound_improved} for $p=1$ were proved to be free of embedded eigenvalues of~$D_V$, Corollary~\ref{cor:emb.EV.free.region} would not be improved. This is also illustrated by Figure~\ref{fig:sp_ell1_compare} (part~a) in the Appendix, where the enclosures provided by Theorem~\ref{thm:1-norm} and Theorem~\ref{thm:spec_bound_improved} are compared.
\end{Remark}

In addition to the statement of Theorem~\ref{thm:spec_bound_improved}, we prove that the improved spectral enclosure for $\ell^{1}$-potentials is at least partly optimal. Namely, we show that a significant part of the boundary of the spectral enclosure is actually an eigenvalue of a concretely chosen discrete Dirac operator within the studied class. This means that this spectral bound cannot be significantly improved.

The proof of the tighter spectral bound of Theorem~\ref{thm:spec_bound_improved} for $\ell^{1}$-potentials does not  make use of majorizing spectral norms by Hilbert--Schmidt norms.
The reason for a possible but unnecessary passing to the Hilbert--Schmidt norms is 
that the resulting spectral bounds are of comparatively simpler forms. 
If we prefer a less sharp but more explicit result for $1<p\leq\infty$, then majorizing
$|T_{i}(k)|\leq|T_{i}(k)|_{\text{HS}}$, for $i=1,2$, and applying natural estimates for $|T_{i}(k)|_{\text{HS}}$, see Lemma~\ref{lem:estim.Green}, we arrive at the following corollary of Theorem~\ref{thm:spec_bound_improved}.	

\begin{Corollary}\label{cor:p-norm}
	Let $1<p\leq\infty$ and assume $V\in\ell^p\bigl(\ZZ,\CC^{2\times2}\bigr)$. 
	If $\la\notin\sigma(D_0)$ satisfies 
	\begin{equation}\label{eq:p-norm}
	f_q(\la,m) \|V\|_p<1
	\end{equation}
	with
	\begin{equation}
	f_q(\la,m):=\frac{|\la-m|+|\la+m|}{\sqrt{|\lambda^2-m^2||\lambda^2 -m^2-4|}}
	\bigg(1+\frac{2\sqrt{|k|^q}}{1-|k|^q}\bigg)^{1/q},
	\end{equation}
	where $k$ is a unique point in the punctured unit disk $\{k\in\CC\;|\;\; 0<|k|<1\}$ such that $\la^{2}=m^2+2-k-k^{-1}$,
	then $\la\notin\sigma(D_V)$.
\end{Corollary}
\begin{Remark}
Note that, if $p\to1$, \ie~$q\to\infty$ in Theorem~\ref{thm:p-norm.stein} and Corollary~\ref{cor:p-norm}, we arrive at the spectral enclosure of Theorem~\ref{thm:1-norm} with the exception of possibly embedded eigenvalues.
\end{Remark}
Stein's interpolation together with the improved spectral 
bound of Theorem~\ref{thm:spec_bound_improved} for the case $p=1$ leads to the following 
improvement of Theorem~\ref{thm:p-norm.stein}.
\begin{Theorem}\label{thm:stein-impr}
Let $1 \leq p\leq\infty$ and assume 
$V\in\ell^p\bigl(\ZZ,\CC^{2\times2}\bigr)$. 
If $\la\notin\sigma(D_0)$ satisfies 
\begin{equation}\label{eq:p-norm_impr}
\psi_q(\la,m) \|V\|_p<1
\end{equation}
then $\la\notin\sigma(D_V)$,
with
\begin{equation}
\psi_q(\la,m):=\left(\max\{|T_{0}(k)|,|T_{1}(k)|\}\right)^{1-\frac{1}{q}}\left(\dist(\lambda,\sigma(D_0))\right)^{-\frac{1}{q}}
\end{equation}
and $k$ the unique point in the punctured unit disk
$\{k\in\CC\;|\;\; 0<|k|<1\}$ such that $\la^{2}=m^2+2-k-k^{-1}$.
The matrices $T_{0}(k)$ and $T_{1}(k)$ are defined
in~\eqref{eq:def_T0_T1} and their norms are given by 
formulas~\eqref{eq:T0_norm} and~\eqref{eq:T1_norm}.
\end{Theorem}
\begin{Remark}\label{rem:diagonal.dominance}
 	One may hope that $|T_{0}(k)|\geq |T_{1}(k)|$ which would mean that 
	the resolvent operator $(D_{0}-\lambda)^{-1}$ is diagonally dominant, see formula~\eqref{eq:res_D0} given below. This would turn the spectral enclosure of Theorem~\ref{thm:spec_bound_improved} (especially in the case of $\ell^{1}$-potentials) into a reasonably simple form. Unfortunately, the inequality $|T_{0}(k)|\geq |T_{1}(k)|$ does \emph{not} hold in general. This can be verified analytically for $m=0$ and therefore the inequality remains false for $m$ small by continuity. Moreover, the dependence of the relation between the values of $|T_{0}(k)|$ and $|T_{1}(k)|$ on the parameter~$k$ seems to be nontrivial, see Figure~\ref{fig:T01ineq} in the Appendix.
\end{Remark}

Similarly as in Theorem~\ref{thm:1-norm}, the spectral enclosures from Theorems~\ref{thm:p-norm.stein}, \ref{thm:spec_bound_improved}, and~\ref{thm:stein-impr} are symmetric with respect to the real as well as the imaginary axes. On the other hand, if $p>1$, these enclosures always contain the entire essential spectrum of $D_V$ for any choice of the potential $V\in\ell^p\bigl(\ZZ,\CC^{2\times2}\bigr)$ and $m>0$. Illustrative plots as well as comparisons of the obtained results are given in the Appendix.

\subsection{Organization of the paper}

As preliminary results for our proofs, in Sections~\ref{Sec.free} and~\ref{Sec.BS}
we recall the resolvent of the free discrete Dirac operator~$D_{0}$
and develop the Birman--Schwinger principle for the operator~$D_{V}$. The proofs of Theorems~\ref{thm:1-norm}-\ref{thm:stein-impr} are presented in Section~\ref{Sec.proofs}.
In Section~\ref{sec:optim}, the optimality of the improved spectral enclosure for $\ell^{1}$-potentials from Theorem~\ref{thm:spec_bound_improved} is discussed.

The paper is concluded by four appendices.
In Appendix~A we numerically visualise the spectral enclosure of Theorem~\ref{thm:p-norm.stein} for several choices of $p>1$. Several comparison plots as well as an illustration of the partial optimality proved for the spectral enclosure from Theorem~\ref{thm:spec_bound_improved} for $p=1$ are given in the parts~B and~C of the Appendix. Finally, Appendix~D serves as a numerical illustration of Remark~\ref{rem:diagonal.dominance}.

\section{The free resolvent}\label{Sec.free}
%
Making use of the observation that 
\[
  D_0^{2}=
	\begin{pmatrix}
	m^{2}+dd^{*} & 0 \\[0.5ex]
	0 & m^{2}+d^{*}d 
	\end{pmatrix} 
\]
together with the familiar formula for the resolvent of the 
discrete Laplacian $dd^{*}=d^{*}d$ (see \eg~\cite[Chp.~1]{teschl00} 
or~\cite[Eq.~(2.2)]{Ibr-Sta-19}), the resolvent of the free Dirac 
operator can be expressed fully explicitly. 
Using the $2\times2$-block matrix representation as 
in~\eqref{free.dirac.matrix}, the resulting formula for the 
resolvent of $D_{0}$ can be written as the $2\times 2$-block 
Laurent matrix
\begin{equation}\label{eq:res_D0}
 \left(D_0-\la\right)^{-1}=\begin{pmatrix}
	 & \vdots & \vdots &  \vdots  & \vdots & \vdots  &   \\
	\dots &  T_{-1}(k)     & T_{0}(k) & T_{1}(k) & T_{2}(k) & T_{3}(k) & \dots      \\
	\dots &  T_{-2}(k)     & T_{-1}(k) & T_{0}(k) & T_{1}(k) & T_{2}(k) & \dots      \\
	\dots &  T_{-3}(k)     & T_{-2}(k) & T_{-1}(k) & T_{0}(k) & T_{1}(k) & \dots      \\
	& \vdots & \vdots &  \vdots  & \vdots & \vdots  &   \\
	\end{pmatrix},
\end{equation}
where $T_{0}(k)$ and $T_{1}(k)$ are defined by~\eqref{eq:def_T0_T1} and 
\begin{equation}\label{Ti}
 T_{i}(k)=T_{-i}^{T}(k)=k^{i-1}T_{1}(k), \;\,\mbox{ for } i\geq1.
\end{equation}
Here $0<|k|<1$ and the spectral parameter~$\la$ is related to $k$ by the equation~\eqref{Jouk}
which determines a one-to-two mapping $k\mapsto\la(k)$ between the punctured 
unit disk $\{k\in\CC\;|\;\; 0<|k|<1\}$ and the resolvent set 
of~$D_{0}$.

For later purpose, we will need the following estimate of the Hilbert--Schmidt norm of the resolvent entries $T_{i}(k)$, $i\in\ZZ$.
\begin{Lemma}\label{lem:estim.Green}
Let $\la\in\rho(D_0)$ and $k$ be such that $0<|k|<1$ related to $\lambda$ via~\eqref{Jouk}. Then one has 
\begin{equation}\label{entrywise.estim}
	|T_j(k)|_{\mathrm{HS}} 
	\leq 
	\frac{|C_j(k)|}{|k^{-1}-k|}\bigl(|\la-m|+|\la+m|\bigr), \quad j\in\ZZ,
\end{equation}
where 
\begin{equation}
	C_j(k)=\begin{cases}
	1 & \mbox{ if } \quad j=0,\\
	k^{|j|-1/2} & \mbox{ if } \quad j\neq0.
	\end{cases}
\end{equation}
%
\end{Lemma}
\begin{proof}
By a straightforward computation, one gets
\begin{equation*}\label{T0.HS.norm}
	 |T_{0}(k)|_{\mathrm{HS}}^{2}=\frac{1}{|k^{-1}-k|^{2}}
	 \left(|\lambda-m|^{2}+|\lambda+m|^{2}+2|k|\left|k^{-1/2}-k^{1/2}\right|^{2}\right)
\end{equation*}
and
\begin{equation*}\label{Ti.HS.norm}
	\begin{aligned}
	 |T_{j}(k)|_{\mathrm{HS}}^{2}&=|T_{-j}(k)|_{\mathrm{HS}}^{2} \\
	 &=\frac{|k|^{2j-1}}{|k^{-1}-k|^{2}}\left(|k||\lambda-m|^{2}+|k||\lambda+m|^{2}
		 +(1+|k|^{2})\left|k^{-1/2}-k^{1/2}\right|^{2}\right),
	\end{aligned}
\end{equation*}
for $j\geq1$. To arrive at~\eqref{entrywise.estim}, it suffices 
to note that, for $0<|k|<1$, one has
\[
 2|k|\left|k^{-1/2}-k^{1/2}\right|^{2}\leq(1+|k|^{2})
 \left|k^{-1/2}-k^{1/2}\right|^{2}\leq2\left|k^{-1/2}-k^{1/2}\right|^{2}
\]
and use the equality
\[
 \left|k^{-1/2}-k^{1/2}\right|^{2}=|\la^2-m^2|
\]
which follows readily from~\eqref{Jouk}.
\end{proof}

\section{The Birman--Schwinger principle}\label{Sec.BS}
%
For $n\in\ZZ$, we denote by $w_n$ the absolute value of~$\ups_n$, 
\ie~$w_n:=\sqrt{\ups^*_n\ups_n}$. 
Using the polar decomposition of matrices, we have  
$
\ups_n=u_nw_n,
$
where $u_n\in\CC^{2\times2}$ is a partial isometry. Notice that 
$|\ups_n|=|w_n|=|\sqrt{w_n}|^2$.
Further, let us denote by $U$ and $W$ the $2\times2$-block diagonal 
matrices with the diagonal block 
entries $u_n$ and $w_n$, respectively, \ie
\[
 U=\bigoplus_{n\in\ZZ}u_{n} \quad\mbox{ and }\quad W=\bigoplus_{n\in\ZZ}w_{n}.
\]
Then $U$ is a partial isometry and we have
\[
V=UW=U\sqrt{W}\sqrt{W},
\]
where $\sqrt{W}$ is the square root of the positive operator~$W$.

Given any $\la\in\rho(D_0)$, we introduce the 
\emph{Birman--Schwinger operator}
\begin{equation}\label{BS.op}
K(\la) := \sqrt{W} \, (D_0-\la)^{-1} \, U\sqrt{W}
\end{equation}
and recall the conventional Birman--Schwinger principle
\begin{equation}\label{BS.disc.EV}
\la\in\sigma(D_V) \quad 
\Longleftrightarrow 
\quad -1\in\sigma(K(\la))
\end{equation}
which can be easily justified by usual arguments 
if, for example, $V$ is bounded.

The next lemma resembles a one-sided 
version of the Birman--Schwinger principle extended 
to possibly embedded eigenvalues. The strategy of the proof we 
provide below is inspired by the ones of analogous results 
in~\cite{Ibr-Kre-Lap, Fan-Kre-LMP19, Fan-Kre-Veg-JST18, Fra-Sim-JST-17}. We recall the notation 
$\cH=\ell^{2}(\ZZ,\CC^{2})$.

\begin{Lemma}\label{Lem.BS}
Let $V\in\ell^1\bigl(\ZZ,\CC^{2\times2}\bigr)$ and let 
$\la\in\sigma(D_0)\setminus\{\pm m, \pm\sqrt{m^2+4}\}
$ 
be such that 
$D_V\psi=\la\psi$ for some $\psi\in\cH$. Then 
$\phi:=\sqrt{W}\psi\in\cH$ and, for all $\varphi\in\cH$, we have
\begin{equation}\label{BS}
	\lim_{\eps \to 0^+} 
	(\varphi,K(\la+\ii\eps)\phi)_{\cH}=-(\varphi,\phi)_{\cH}.
\end{equation}
\end{Lemma}
\begin{proof}
It is not difficult to check that
\begin{equation}
	\|\sqrt{W}\|_{\cH\to\cH} \leq \sqrt{\|V\|_1}.
\end{equation}
Hence $\phi=\sqrt{W}\psi\in\cH$.

Let $\varphi\in\cH$ be fixed and $\eps>0$ be arbitrary.
Then $\la + \ii\eps \notin\sigma(D_0)$ and we have
\begin{equation*}
\begin{aligned}
	\bigl(\varphi,K(\la+\ii\eps)\phi\bigr)_{\cH} 
	&=\bigl(\varphi, \sqrt{W}(D_0-\la-\ii\eps)^{-1}V\psi\bigr)_{\cH}\\ 
	&=\bigl(\varphi, \sqrt{W}(D_0-\la-\ii\eps)^{-1}\bigl(-(D_0-\la-\ii\eps)\psi-\ii\eps\psi\bigr)\bigr)_{\cH}\\ 
	&=-\bigl(\varphi, \phi\bigr)_{\cH}-\ii\eps\bigl(\varphi, \sqrt{W}(D_0-\la-\ii\eps)^{-1}\psi\bigr)_{\cH}.
\end{aligned}
\end{equation*}
Further, denoting $M(\eps):=\sqrt{W}(D_0-\la-\ii\eps)^{-1}$
and employing the Cauchy--Schwarz inequality, we observe 
\begin{equation*}
	\Bigl|\bigl(\varphi, \sqrt{W}(D_0-\la-\ii\eps)^{-1}\psi\bigr)_{\cH}\Bigr|
	\leq \|\varphi\|_{\cH} \, \|M(\eps)\|\, \|\psi\|_{\cH}
	\,.
\end{equation*}
In the remaining part of the proof, we show that $\eps\|M(\eps)\|\to0$ as 
$\eps \to 0^+$ from which the claim will follow.
Let $k=k(\eps)$ be the unique point inside the unit disk corresponding to 
$\la+\ii\eps$ via~\eqref{Jouk}, where $\la$ is replaced by $\la+\ii\eps$. 
Applying Lemma~\ref{lem:estim.Green}, we get the estimate
\begin{align*}
 \left|k-k^{-1}\right|^2\sum_{i\in\ZZ}|T_{i}(k)|_{HS}^{2}&=
 \left|k-k^{-1}\right|^2\left(|T_{0}(k)|_{\text{HS}}^{2}+2\sum_{i=1}^{\infty}|T_{i}(k)|_{\text{HS}}^{2}\right)\\
 &\leq\bigl(|\la+\ii\eps-m|+|\la+\ii\eps+m|\bigr)^2\left(1+2\sum_{i=1}^{\infty}|k|^{2i-1}\right)\\
 &\leq\frac{2}{1-|k|^{2}}\bigl(|\la+\ii\eps-m|+|\la+\ii\eps+m|\bigr)^2.
\end{align*}
Therefore, 
\begin{equation*}
	\begin{aligned}
	\|M(\eps)\|^2 \leq \|M(\eps)\|_\mathrm{HS}^2
	&=\sum_{i\in\ZZ}\sum_{j\in\ZZ}\bigl|\sqrt{w_i}T_{j-i}(k)\bigr|_\mathrm{HS}^2
	\leq\sum_{i\in\ZZ}\sum_{j\in\ZZ}\bigl|\sqrt{w_i}\bigr|^2\bigl|T_{j-i}(k)\bigr|_\mathrm{HS}^2\\
	&=\sum_{i\in\ZZ}|\ups_i|\sum_{j\in\ZZ}\bigl|T_{j}(k)\bigr|_\mathrm{HS}^2\\
	&\leq \frac{2}{|k-k^{-1}|^2}\frac{\|V\|_1}{1-|k|^2}\bigl(|\la+\ii\eps-m|+|\la+\ii\eps+m|\bigr)^2.
	\end{aligned}
\end{equation*}
For $\la\in(-\sqrt{m^2+4},-m)\cup(m,\sqrt{m^2+4})$, elementary calculations show that
\begin{equation*}
	\frac{1}{|k-k^{-1}|^2}\frac{1}{1-|k|^2}
	=\mathcal{O}(\eps^{-1})
	\quad \mbox{as} \quad \eps\to0^+.
\end{equation*}
Hence, $\eps\|M(\eps)\|$ decays at least as $\mathcal{O}(\eps^{1/2})$ 
for $\eps\to0^+$.
\end{proof}
%
\section{Proofs}\label{Sec.proofs}
%
\subsection{Proof of Theorem~\ref{thm:1-norm}}
First we consider the case $\la\notin[-\sqrt{m^2+4},-m]\cup[m,\sqrt{m^2+4}]$.
Let $k\in\CC$ denote the point inside the punctured unit disk determined by $\lambda$ via~\eqref{Jouk}.

%

For $V\in\ell^1\bigl(\ZZ,\CC^{2\times2}\bigr)$, the 
Birman--Schwinger operator $K(\la)$ is Hilbert--Schmidt.
To estimate its Hilbert--Schmidt norm we use the general inequality
\[
 \|ABC\|_{\rm{HS}}\leq\|A\|\|B\|_{\rm{HS}}\|C\|,
\]
which holds true for any Hilbert-Schmidt operator $B$ and bounded operators $A,C$; see, \eg~\cite[Prop.~IV.2.3]{Goh-Gol-Kru00}. 
Moreover, using that $|C_{j}(k|)\leq1$ and
\begin{equation}
 \left|k-k^{-1}\right|^2=\left|\la^2-m^2\right|\left|\la^2-m^2-4\right|
\label{eq:k_minus_k_recip_in_lam}
\end{equation}
by~\eqref{Jouk}, in Lemma~\ref{lem:estim.Green}, we obtain
\[
 	\bigl|T_{i}(k)\bigr|^2_\mathrm{HS}
	\leq
	\frac{(|\la-m|+|\la+m|)^{2}}{|\la^2-m^2-4||\la^2-m^2|}.
\]
Now, we may estimate the Hilbert--Schmidt norm of $K(\la)$ as follows:
\begin{equation*}
\begin{aligned}
	\|K(\la)\|^2_{\mathrm{HS}}
	&=\sum_{i\in\ZZ}\sum_{j\in\ZZ}\left|\sqrt{w_i}T_{j-i}(k)u_j\sqrt{w_j}\right|_{\mathrm{HS}}^{2}\\ 
	&\leq\sum_{i\in\ZZ}\sum_{j\in\ZZ}\left|\sqrt{w_i}\right|^{2}\left|T_{j-i}(k)\right|_{\mathrm{HS}}^{2}\left|\sqrt{w_j}\right|^{2}\\ 
	&=\sum_{i\in\ZZ}\sum_{j\in\ZZ}|\ups_i|\left|T_{j-i}(k)\right|_{\mathrm{HS}}^{2}|\ups_j|\\
	&\leq\|V\|^2_{1}\frac{(|\la-m|+|\la+m|)^{2}}{|\la^2-m^2-4||\la^2-m^2|}.
\end{aligned}
\end{equation*}
Therefore, 
\begin{equation}\label{estim.BS.case.1}
	\|K(\la)\|\leq\|K(\la)\|_\mathrm{HS}\leq\|V\|_{1}\frac{|\la-m|+|\la+m|}{\sqrt{|\la^2-m^2-4||\la^2-m^2|}}.
\end{equation}
and the Birman--Schwinger principle \eqref{BS.disc.EV} implies that 
$\la$ cannot belong to the point spectrum of $H_V$ unless 
it holds that 
\begin{equation}\label{EV.bd.prf}
	|\la^2-m^2||\la^2-m^2-4|\leq \|V\|^{2}_{1}(|\la-m|+|\la+m|)^{2}.
\end{equation}

Now we consider the case $\la\in(-\sqrt{m^2+4},-m) \cup (m,\sqrt{m^2+4})$.
Let $\eps>0$ be arbitrary.
Since $\la+\ii\eps\notin[-\sqrt{m^2+4},-m] \cup [m,\sqrt{m^2+4}]$, we can 
apply \eqref{estim.BS.case.1}	
and deduce
\begin{equation}\label{estim.BS.case.2}
	\|K(\la+\ii\eps)\| \leq \|V\|_1\frac{|\la+\ii\eps-m|+|\la+\ii\eps+m|}{\sqrt{|(\la+\ii\eps)^2-m^2-4||(\la+\ii\eps)^2-m^2|}}.
\end{equation}
On the other hand, if $\la\in\sigma_{\rm{p}}(D_V)$ with an eigenvector 
$\psi\in\cH$, then we can invoke Lemma~\ref{Lem.BS} and apply \eqref{BS} with 
$\varphi=\phi=W^{1/2}\psi$. Taking the 
limit~$\eps\to0^+$, we thus obtain 
\begin{equation}\label{BS.2.imp}
	\|\varphi\|^2
	\leq 
	\liminf_{\eps\to 0^+} \|K(\la+\ii\eps)\| \|\varphi\|^2
	\,.
\end{equation}
However, it is not difficult to see that $\varphi=W^{1/2}\psi\neq0$ 
(otherwise, $\la$ would be an eigenvalue 
of~$D_0$ which is impossible). Hence, we deduce 
from \eqref{BS.2.imp} that $\liminf_{\eps\to 0^+} \|K(\la+\ii\eps)\|\geq1$. 
Therefore, letting $\eps\to0^+$ in \eqref{estim.BS.case.2}, we
conclude
\begin{equation}\label{eq:not.improvable}
  1\leq \liminf_{\eps\to 0^+} \|K(\la+\ii\eps)\|
  \leq
  \|V\|_{1}\frac{|\la-m|+|\la+m|}{\sqrt{|\la^2-m^2-4||\la^2-m^2|}}.
\end{equation}
\ie~also embedded eigenvalues must obey the estimate~\eqref{EV.bd.prf}.

Finally, since the endpoints $\lambda\in\{\pm\sqrt{m^{2}+4},\pm m\}$ are involved in the set on the right-hand side of~\eqref{sp.enc.l1} the proof is completed.
\qed

\subsection{Proof of Theorem~\ref{thm:p-norm.stein}}
The proof is based on the following special variant of Stein's complex interpolation theorem, see~\cite[Thm.~1]{Ste-TAMS-1956}. 
\begin{Lemma}[Stein's interpolation]
	\label{lem:stein}
Let $T_{z}:\ell^{2}(\ZZ;\CC^{2\times2})\to\ell^{2}(\ZZ;\CC^{2\times2})$ 
be a family of operators analytic in the strip $0<\Re z<1$ and 
continuous and uniformly bounded in its closure $0\leq\Re z\leq1$. 
Suppose further that there exist constants $C_{0}$ and $C_{1}$ such that
	\[
	\|T_{\ii y}\|\leq C_{0} \quad\mbox{ and }\quad \|T_{1+\ii y}\|\leq C_{1},
	\]
for all $y\in\RR$. Then, for any $\theta\in[0,1]$, one has
	\[
	\|T_{\theta}\|\leq C_{0}^{1-\theta}C_{1}^{\theta}.
	\]
\end{Lemma}
\begin{proof}[Proof of Theorem~\ref{thm:p-norm.stein}]
Let 
$\la\notin\sigma(D_0)$ be fixed. If $p=\infty$, one has the trivial estimate for the Birman--Schwinger operator 
	\[
	\|K(\la)\|\leq\|\sqrt{W}\|\|(D_{0}-\la)^{-1}\|\|\sqrt{W}\|\leq\frac{\|V\|_{\infty}}{\dist(\la,\sigma(D_0))}.
	\]
	Then~\eqref{BS.disc.EV} implies~\eqref{eq:p-norm.2} in the particular case $p=\infty$.

	For the case $1<p<\infty$, we consider the operator family 
	\[
	T_{z}:=W^{z p/2}(D_0-\la)^{-1}W^{z p/2},
	\]
	for $z\in\CC$ with $0\leq\Re z\leq 1$. Note that $T_{z}$ is continuous in the closed strip $0\leq\Re z\leq 1$ and analytic in its interior. Moreover, $T_{z}$ is uniformly bounded for $0\leq\Re z\leq 1$ as one has 
	\[
	\sup_{0\leq\Re z\leq 1}\|T_{z}\|\leq\frac{\max(1,\|V\|_{\infty}^{p})}{\dist(\la,\sigma(D_0))}.
	\]
	Further, since $V\in\ell^p(\ZZ,\CC^{2\times2})$ by the hypothesis, we can apply~\eqref{estim.BS.case.1} to get
	\[
	\|T_{1+\ii y}\|\leq\|W^{p/2}(D_0-\la)^{-1}W^{p/2}\| \leq 
	\frac{|\la-m|+|\la+m|}{\sqrt{\left|\la^2-m^2\right|\left|\la^2-m^2-4\right|}}\|V\|^p_p,
	\]
	for any $y\in\RR$. Moreover, for all $y\in\RR$, we have also the estimate
	\[
	\|T_{\ii y}\|\leq \frac{1}{\dist(\la,\sigma(D_0))}.
	\]
	Thus, one can apply Theorem~\ref{lem:stein} with $\theta=1/p$ which implies
	\begin{equation*}
	\|K(\la)\|\leq\|T_{1/p}\| \leq 
	\frac{(|\la-m|+|\la+m|)^{\frac{1}{p}}\|V\|_p}{|\lambda^2-m^2|^{\frac{1}{2p}}{|\lambda^2 -m^2-4|}^{\frac{1}{2p}}}
	\frac{1}{(\dist(\la, \sigma(D_0)))^{1-\frac{1}{p}}}
	\end{equation*}
	and the claim immediately follows from the Birman--Schwinger principle~\eqref{BS.disc.EV}.
\end{proof}
\subsection{Proof of Theorem~\ref{thm:spec_bound_improved}}
In the proof of Theorem~\ref{thm:spec_bound_improved}, will need a discrete version of Young's inequality.

\begin{Lemma}[Young's inequality]\label{lem:young}
	Let $p,q,r \geq 1$ be such that  
	\[
	\frac{1}{p}+\frac{1}{q}+\frac{1}{r}=2.
	\]
	Then, for any $f\in\ell^{p}(\ZZ)$, $g\in\ell^{q}(\ZZ)$, and $h\in\ell^{r}(\ZZ)$, one has
	\begin{equation}
	\sum_{i,j\in\ZZ}|f_{i}||g_{j-i}||h_{j}|\leq \|f\|_{p}\|g\|_{q}\|h\|_{r}.
	\label{eq:young}
	\end{equation}
	Moreover, the inequality~\eqref{eq:young} is sharp.
\end{Lemma}

Young's inequality holds true in a very abstract setting
see, \eg~\cite[Thm.~20.18]{Hew-Ros79}, which implies Lemma~\ref{lem:young} as a particular case.
Differently from the continuous setting, \cf~\cite[Sec.~4.2]{Lie-Los01}, the optimal constant in the inequality \eqref{eq:young} is $1$, indeed. One can prove the discrete variant of Young' inequality by mimicking the standard arguments used typically for the proof of the continuous variant of the inequality. We provide this proof of Theorem~\ref{lem:young} for reader's convenience.

\begin{proof}[Proof of Lemma~\ref{lem:young}]
 Let $f\in\ell^{p}(\ZZ)$, $g\in\ell^{q}(\ZZ)$ $h\in\ell^{r}(\ZZ)$, for $p,q,r\geq1$ such that 
 \[
  \frac{1}{p}+\frac{1}{q}+\frac{1}{r}=2.
 \]
 Observe that
 \[
 \sum_{i,j\in\ZZ}|f_{i}||g_{j-i}||h_{j}|=\|\varphi\chi\psi\|_{1},
 \]
 for
 \begin{align*}
  \varphi_{i,j}&:=|f_{i}|^{p\left(1-\frac{1}{r}\right)}|g_{j-i}|^{q\left(1-\frac{1}{r}\right)},\\
  \chi_{i,j}&:=|f_{i}|^{p\left(1-\frac{1}{q}\right)}|h_{j}|^{r\left(1-\frac{1}{q}\right)},\\
  \psi_{i,j}&:=|g_{j-i}|^{q\left(1-\frac{1}{p}\right)}|h_{j}|^{r\left(1-\frac{1}{p}\right)}.
 \end{align*}
 Moreover, $\varphi\in\ell^{u}(\ZZ^{2})$, $\chi\in\ell^{v}(\ZZ^{2})$, $\varphi\in\ell^{w}(\ZZ^{2})$, where $u,v,w\geq1$ are such that 
 \[
 \frac{1}{u}+\frac{1}{r}=1, \quad \frac{1}{v}+\frac{1}{q}=1, \quad \frac{1}{w}+\frac{1}{p}=1.
 \]
 By the assumptions, the indices $u,v,w$ fulfill
 \[
  \frac{1}{u}+\frac{1}{v}+\frac{1}{w}=1.
 \]
 and the application of (generalized) H\"{o}lder's inequality yields
 \begin{align*}
  \sum_{i,j\in\ZZ}|f_{i}||g_{j-i}||h_{j}|&\leq\|\varphi\|_{u}\|\chi\|_{v}\|\psi\|_{w}
  =\left(\|f\|_{p}^{p}\|g\|_{q}^{q}\right)^{1-\frac{1}{r}}
    \left(\|f\|_{p}^{p}\|h\|_{r}^{r}\right)^{1-\frac{1}{q}}
    \left(\|g\|_{q}^{q}\|h\|_{r}^{r}\right)^{1-\frac{1}{p}}\\
    &=\|f\|_{p}\|g\|_{q}\|h\|_{r}.
 \end{align*}
 
  The constant $1$ on the right-hand side of~\eqref{eq:young} is optimal,
  indeed the equality is attained for $f=\{f_n\}_{n\in\ZZ}$, $g=\{g_n\}_{n\in\ZZ}$, and
  $h=\{h_n\}_{n\in\ZZ}$ with entries
  \[
    f_n = f_0\delta_{n,0}, \quad g_n =g_0\delta_{n,0},\quad\mbox{ and }\quad h_n = h_0\delta_{n,0},
  \]
  for any $f_0, g_0, h_0 \in \Com$.
\end{proof}

\begin{proof}[Proof of Theorem~\ref{thm:spec_bound_improved}]
	Assume $\lambda\in\sigma(D_V)\setminus\sigma(D_0)$ and $k\in\CC$, $0<|k|<1$ is such that~\eqref{Jouk} holds. The main idea of the proof relies again on the implication from the Birman--Schwinger principle:
	\[
	\lambda\in\sigma(D_V)\setminus\sigma(D_0) \quad\Longrightarrow\quad \|K(\lambda)\|\geq1.
	\]
	
	The case $V\in\ell^{1}(\ZZ,\CC^{2\times2})$: Using the definition of the Birman--Schwinger operator~\eqref{BS.op}, we may estimate
	\begin{align*}
	\|K(\la)\phi\|^{2}_{2}&=\sum_{i\in\ZZ}\bigg|\sum_{j\in\ZZ}K_{i,j}(\la)\phi_{j}\bigg|_{2}^{2}\leq\sum_{i\in\ZZ}\left(\sum_{j\in\ZZ}\left|\sqrt{w_{i}}\right||T_{j-i}(k)|\left|\sqrt{w_{j}}\right||\phi_{j}|_{2}\right)^{\!2}\\
	&\leq\sup_{n\in\ZZ}|T_{n}(k)|^{2}\sum_{i\in\ZZ}|\ups_{i}|\left(\sum_{j\in\ZZ}\left|\sqrt{w_{j}}\right||\phi_{j}|_{2}\right)^{\!2}\\
	&\leq\sup_{n\in\ZZ}|T_{n}(k)|^{2}\|V\|_{1}^{2}\|\phi\|_{2}^{2},
	\end{align*}
	for $\phi\in\ell^{2}(\ZZ,\Com^{2})$. Here $|\cdot|_{2}$ stands for the Euclidian norm on~$\CC^{2}$. Moreover, it follows from~\ref{Ti} that
	\[
	|T_{i}(k)|\leq|T_{1}(k)|,
	\]
	for all $i\neq0$. Consequently, we get
	\begin{equation}
	\|K(\la)\|\leq\max\{|T_{0}(k)|,|T_{1}(k)|\}\|V\|_{1}
	\label{eq:ineq_K_max}
	\end{equation}
	and the Birman--Schwinger principle implies the claim for the case $p=1$.
	
	The case $V\in\ell^{p}(\ZZ,\CC^{2\times2})$ with $p>1$: Making use of~\eqref{BS.op}, one obtains, for $\phi,\psi\in\ell^{2}(\ZZ,\Com^{2})$, the estimate
	\[
	|\left(\phi,K(\lambda)\psi\right)_{\ell^{2}}|\leq\sum_{i,j\in\ZZ}\left|\sqrt{w_{i}}\right||\phi_{i}|_{2}\left|T_{j-i}(k)\right|\left|\sqrt{w_{j}}\right||\psi_{j}|_{2},
	\]
	The right-hand side is in a suitable form for the application of discrete Young's inequality. Thus, applying~\eqref{eq:young} with $p$ and $r$ replaced by $2p/(p+1)$ and $q$ the H{\" o}lder dual index to $p$, \ie~$q=p/(p-1)$, one arrives at the estimate
	\[
	|\left(\phi,K(\lambda)\psi\right)_{\ell^{2}}|\leq 
	\left[\sum_{i,j\in\ZZ}\left(\left|\sqrt{w_{i}}\right||\phi_{i}|_{2}\right)^{\frac{2p}{p+1}}\right]^{\frac{p+1}{2p}}
	\left\|\left|T(k)\right|\right\|_{q}
	\left[\sum_{i,j\in\ZZ}\left(\left|\sqrt{w_{i}}\right||\phi_{i}|_{2}\right)^{\frac{2p}{p+1}}\right]^{\frac{p+1}{2p}}
	\]
	where $|T(k)|$ stands for the doubly infinite sequence with entries $|T_{i}(k)|$, $i\in\ZZ$. Noticing that, by H{\" o}lder's inequality, 
	\[
	\left[\sum_{i,j\in\ZZ}\left(\left|\sqrt{w_{i}}\right||\phi_{i}|_{2}\right)^{\frac{2p}{p+1}}\right]^{\frac{p+1}{2p}}\leq\left\|\sqrt{W}\right\|_{2p}\left\|\phi\right\|_{2}=\|V\|_{p}^{1/2}\|\phi\|_{2}
	\]
	and
	\begin{align*}
	\left\|\left|T(k)\right|\right\|_{q}&=\left(\sum_{i\in\ZZ}|T_{i}(k)|^{q}\right)^{1/q}=\left(|T_{0}(k)|^{q}+2|T_{1}(k)|^{q}\sum_{i=0}^{\infty}|k|^{qi}\right)^{1/q}\\
	&=\left(|T_{0}(k)|^{q}+\frac{2}{1-|k|^{q}}|T_{1}(k)|^{q}\right)^{1/q},
	\end{align*}
	where we have used~\eqref{Ti}, we obtain
	\[
	|\left(\phi,K(\lambda)\psi\right)_{\ell^{2}}|\leq 
	\left(|T_{0}(k)|^{q}+\frac{2}{1-|k|^{q}}|T_{1}(k)|^{q}\right)^{1/q}\|V\|_{p}\|\phi\|_{2}\|\psi\|_{2}
	\]
	for any $\phi,\psi\in\ell^{2}(\ZZ,\Com^{2})$. In other words, we have the estimate
	\[
	\|K(\lambda)\|\leq 
	\left(|T_{0}(k)|^{q}+\frac{2}{1-|k|^{q}}|T_{1}(k)|^{q}\right)^{1/q}\|V\|_{p}
	\]
	and the Birman--Schwinger principle implies the claim for the case $p>1$.
\end{proof}
\subsection{Proof of Theorem~\ref{thm:stein-impr}}
The proof is completely analogous to that of 
Theorem~\ref{thm:p-norm.stein}. It follows from the application of 
Stein's interpolation theorem using this time the 
inequality~\eqref{eq:ineq_K_max} which plays the same role as
the inequality~\eqref{estim.BS.case.1} 
in the proof of Theorem~\ref{thm:p-norm.stein}.

\section{Optimality for $\ell^1$-potentials}\label{sec:optim}
In~\cite{Ibr-Sta-19}, a similar approach as in the proof of
Theorem~\ref{thm:1-norm} was used to deduce a spectral enclosure for the discrete Schr{\" o}dinger operator with an $\ell^{1}$-potential. In this case, the obtained spectral enclosure turned out to be \emph{optimal} in the following sense: every point from the boundary curve of the spectral enclosure except possible intersections with the spectrum of the unperturbed operator is an eigenvalue of some discrete Schr{\" o}dinger operator with particularly chosen $\ell^{1}$-potential. It means that the obtained spectral enclosure is, in a sense, the best possible since it cannot be further squeezed.

On the contrary, the spectral enclosure of Theorem~\ref{thm:1-norm}
for the discrete Dirac operators with $\ell^{1}$-potentials is
\emph{not optimal} in the aformentioned sense. 
Concerning the optimality of the improved spectral enclosure of Theorem~\ref{thm:spec_bound_improved} in the case of $\ell^{1}$-potential we were not able to prove it in its full generality as specified above. In other words, using the notation of~Theorem~\ref{thm:spec_bound_improved} and denoting the boundary curve as
\[
 \Gamma_{Q}:=\{\lambda\in\Com \mid h_{\infty}(\lambda,m)Q=1\},
\]
for a fixed parameter $Q>0$, we do not have a proof demonstrating that \emph{every} point $\lambda\in\Gamma_{Q}\setminus\sigma(D_0)$ is an eigenvalue of some~$D_{V}$ with $\|V\|_{1}=Q$. However, we can show a sort of a~\emph{partial optimality}. It means that we can show that at least some points from~$\Gamma_{Q}$ are eigenvalues of~$D_{V}$ with a particularly chosen potential~$V\in\ell^1(\ZZ, \Com^{2\times 2})$. These particular points have to be additionally included in the region
\[
  \mathcal{D}:=\{\lambda \in \Com \mid
  |T_{0}(k)|\geq |T_{1}(k)| \mbox{ for } k\in\Com,\, 0<|k|<1, \mbox{such that~\eqref{Jouk} holds}\},
\]
\ie~in the region where the diagonal dominance of the resolvent operator $(D_0-\lambda)^{-1}$ actually happens.
In the proof below, we construct explicitly a potential $V\in\ell^1(\ZZ, \Com^{2\times 2})$, $\|V\|_{1}=Q$, such that $\lambda\in\Gamma_{Q}\cap\mathcal{D}$ not belonging to $\sigma(D_{0})$ is an eigenvalue of~$D_V$.

\begin{Theorem}\label{thm:partial_optim}
  For every $Q>0$ and $\lambda\in\Gamma_{Q}\cap\mathcal{D}\setminus\sigma(D_0)$, there exists $V\in \ell^1(\Int,\Com^{2\times 2})$ with 
  $\norm{V}_1=Q$ such that   
  \[
   \lambda\in\sigma_{p}(D_V).
  \]
\end{Theorem}

\begin{proof}
  Let $Q>0$ and $\lambda\in\Gamma_{Q}\cap\mathcal{D}\setminus\sigma(D_0)$ be fixed. Denote by $k$ the unique point such that $0<|k|<1$ and related to~$\lambda$ by~\eqref{Jouk}.

  First, note that since $\lambda\in\Gamma_{Q}\cap\mathcal{D}\setminus\sigma(D_0)$ one has $h_{\infty}(\lambda,m)=|T_0(k)|=Q^{-1}$. We define the potential sequence $V:=\{\ups_n\}_{n \in \Int}$ entry-wise as follows:
  \[
    \ups_n:=-\delta_{n,0}\,Q^{2}\,T_{0}^{*}(k).
  \]
  Then clearly 
  \[
   \|V\|_{1}=|\ups_{0}|=Q^{2}|T_0(k)|=Q.
  \]
 
  Next, observe that the eigenvalue equation $D_{V}\psi=\lambda\psi$ has a nontrivial solution~$\psi\in\ell^{2}(\ZZ,\Com^{2})$ if and only if there exists a nontrivial vector $\phi\in\ell^{2}(\ZZ,\Com^{2})$ such that
  \begin{equation}\label{eq:eigenvector.resolvent.side}
    \phi = - V(D_{0}-\lambda)^{-1} \phi.
  \end{equation}
  A nontrivial solution~$\phi\in\ell^{2}(\ZZ,\Com^{2})$ of the equation~\eqref{eq:eigenvector.resolvent.side} can be chosen so that $\phi_{n}=0$ for all $n\neq0$ and $\phi_{0}\in\Com^{2}$ is a non-trivial solution of the linear system
  \[
    \left(1+\ups_0 T_0(k)\right)\phi_0 = 0
  \]
  which exists since the matrix $1+\ups_{0}T_{0}(k)$ is singular due to the particular choice of the potential~$V$.
\end{proof}
\begin{Remark}
The construction in Theorem~\ref{thm:partial_optim} is inspired 
by~\cite{Cuenin-Laptev-Tretter_2014} where a sharp spectral 
enclosure was obtained for the one-dimensional Dirac operator
on the real line.
In the continuous setting, the free resolvent has a diagonally 
dominant kernel for every spectral parameter~$\lambda$ from the 
resolvent set and thanks to this $\delta$--potentials can 
serve as an example for proving optimality of the whole 
spectral enclosure. 
The analysis of the discrete setting is not completely analogous, since diagonal dominance of the resolvent only happens in the subdomain $\mathcal{D} \subsetneq \Com$ (see Remark~\ref{rem:diagonal.dominance}).
\end{Remark}

A comparison of the spectral enclosures from Theorems~\ref{thm:1-norm}
and~\ref{thm:spec_bound_improved} for the case of $\ell^{1}$-potentials 
is numerically illustrated in Figure~\ref{fig:sp_ell1_compare} of Appendix~B.
This Figure also shows the parts of the boundary curves of the improved spectral bound of Theorem~\ref{thm:spec_bound_improved} in the case of $\ell^{1}$-potentials that can be reached by an eigenvalue of a concretely chosen discrete Dirac operator as discussed in Theorem~\ref{thm:partial_optim}.

\section*{Appendix. Illustrative and comparison plots}

\subsection*{A: Plots of the spectral enclosures from Theorem~\ref{thm:p-norm.stein}}

The spectral enclosure for the $\ell^{1}$-potentials from Theorem~\ref{thm:1-norm} was displayed already in the introduction in Figure~\ref{fig:ell1plots}. Similarly, we provide several plots illustrating the spectral enclosures from Theorem~\ref{thm:p-norm.stein} in Figure~\ref{fig:sp_thm2} below. Namely, the plots show the boundary curves given by the equation
\[
g_p(\lambda,m) \|V\|_p=1,
\]
for $m=1$, $\|V\|_p=\frac{j}{4}$, $j\in\{1,2,\dots,7\}$, and four choices of $p\in\{3/2,2,3,5\}$.

\begin{figure}[htb!]
    \centering
    \begin{subfigure}[c]{0.49\textwidth}
	\includegraphics[width=\textwidth]{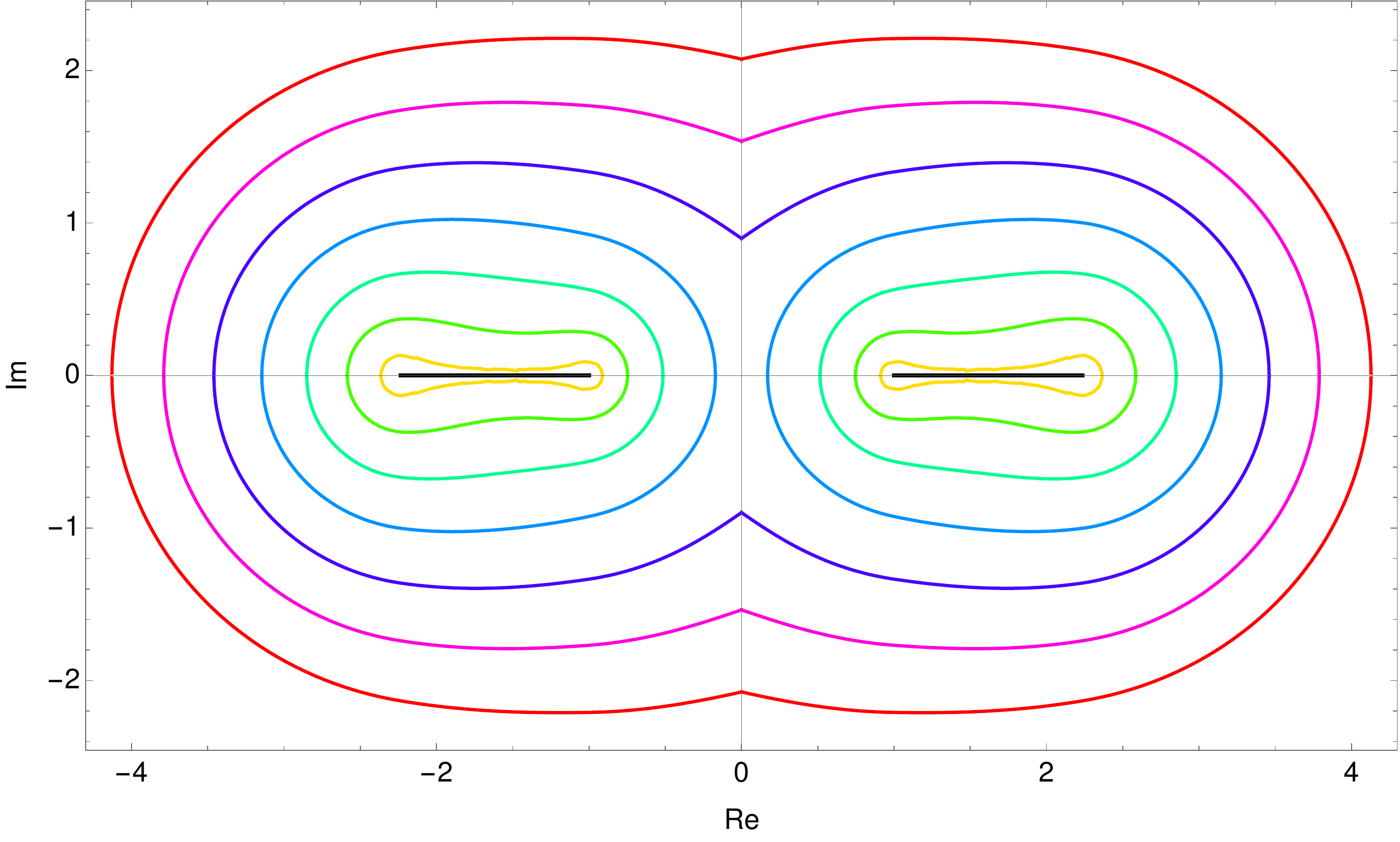}
    \caption{$p=3/2$}
    \end{subfigure}
    \begin{subfigure}[c]{0.49\textwidth}
	\includegraphics[width=\textwidth]{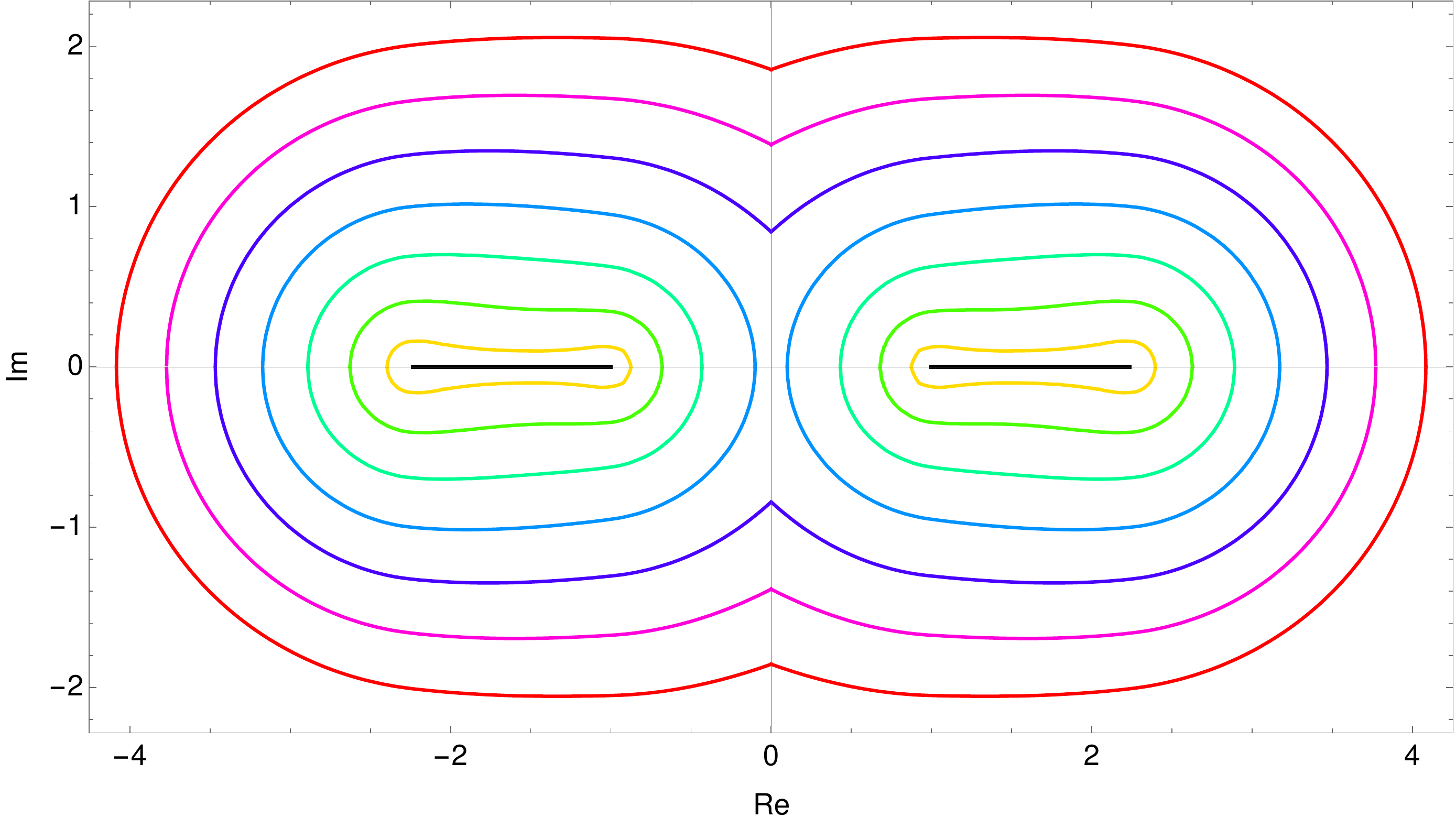}
	\caption{$p=2$}
    \end{subfigure}
    \vskip6pt
    \begin{subfigure}[c]{0.49\textwidth}
	\includegraphics[width=\textwidth]{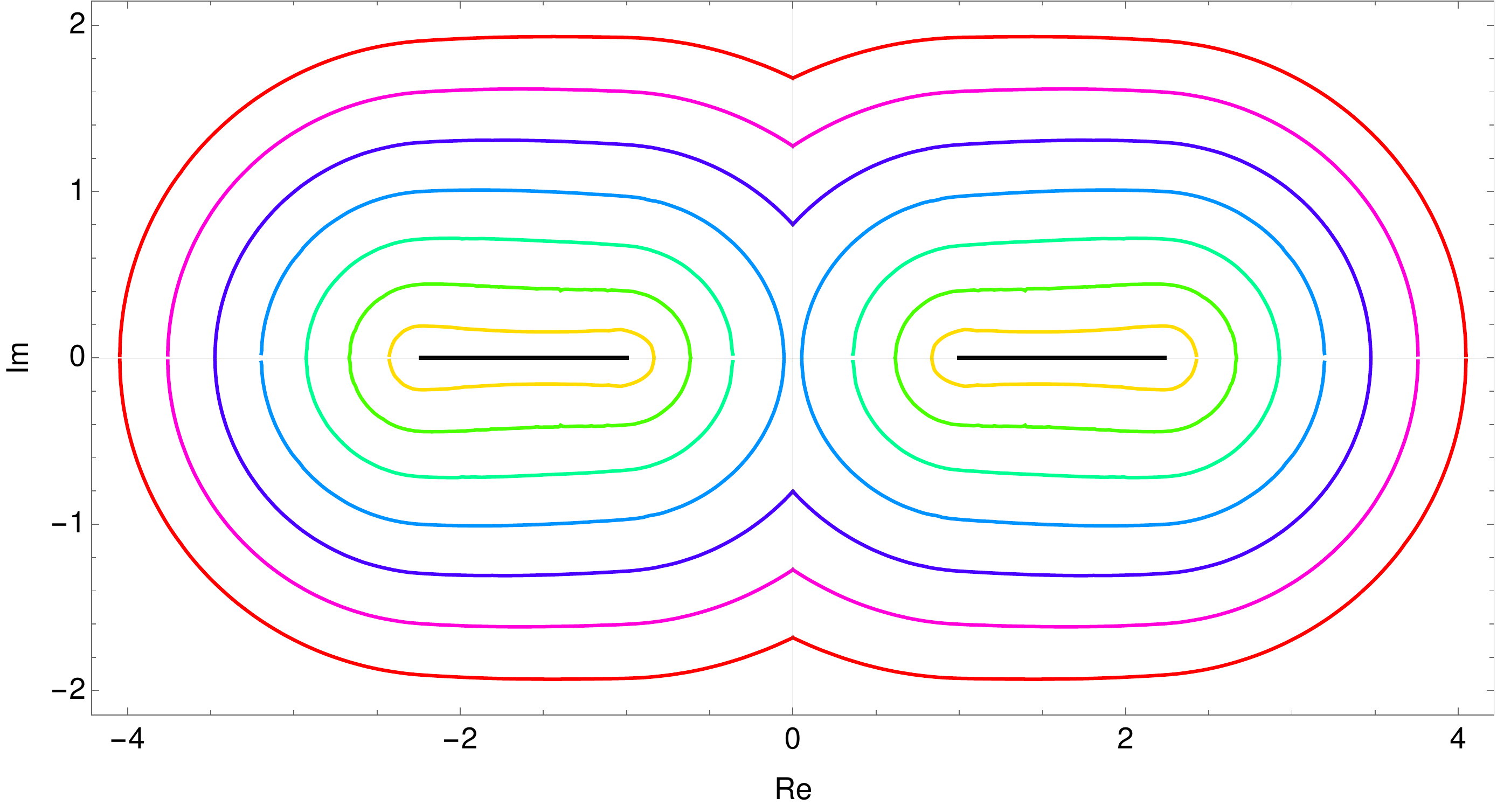}
	\caption{$p=3$}
    \end{subfigure}
    \begin{subfigure}[c]{0.49\textwidth}
	\includegraphics[width=\textwidth]{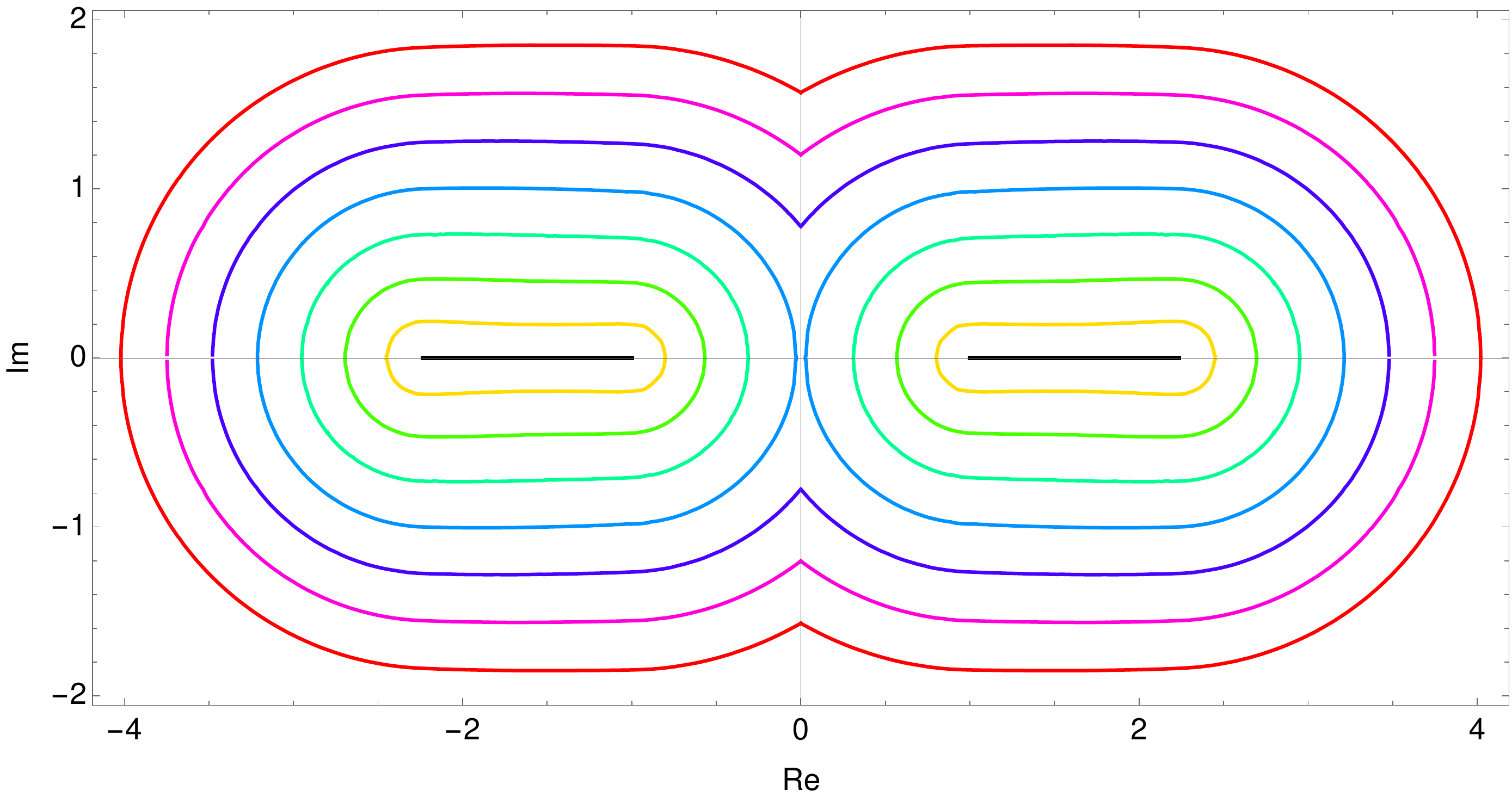}
	\caption{$p=5$}
    \end{subfigure}
    \caption{The plots of the expanding boundary curves 
			corresponding to the spectral enclosure from Theorem~\ref{thm:p-norm.stein}
			for $m=1$, $\|V\|_p=\frac{j}{4}$, $j\in\{1,2,\dots,7\}$, and four choices of the parameter~$p>1$.}
    \label{fig:sp_thm2}
\end{figure}

\subsection*{B: Comparison plots for the $\ell^{1}$-bounds of Theorems~\ref{thm:1-norm} and~\ref{thm:spec_bound_improved} and optimality}

Next set of plots show the boundary curve of the spectral enclosure from Theorem~\ref{thm:1-norm} together with the corresponding improved result of Theorem~\ref{thm:spec_bound_improved} for a comparison. Moreover, the boundary curve of the improved spectral enclosure is made in two colors distinguishing the parts that are eigenvalues of some discrete Dirac operators as discussed in Theorem~\ref{thm:partial_optim}. 

More concretely, in Figure~\ref{fig:sp_ell1_compare}, we plot the boundary curve of Theorem~\ref{thm:1-norm} by \emph{blue dashed lines} for $m=1/2$ and several choices of~$\|V\|_{1}$. At the same time, we add a plot of the curve defined by the equation 
\[
 \max\{|T_{0}(k)|,|T_{1}(k)|\}\|V\|_{1}=1,
\]
by \emph{red or green solid lines} for the same choice of parameters. The parts of the curve made in green belong to the set~$\mathcal{D}$ and hence these points are eigenvalues of some discrete Dirac operators with $\ell^{1}$-potentials. The remaining parts are made in red.

\begin{figure}[H]
    \centering
    \begin{subfigure}[c]{0.9\textwidth}
	\includegraphics[width=\textwidth]{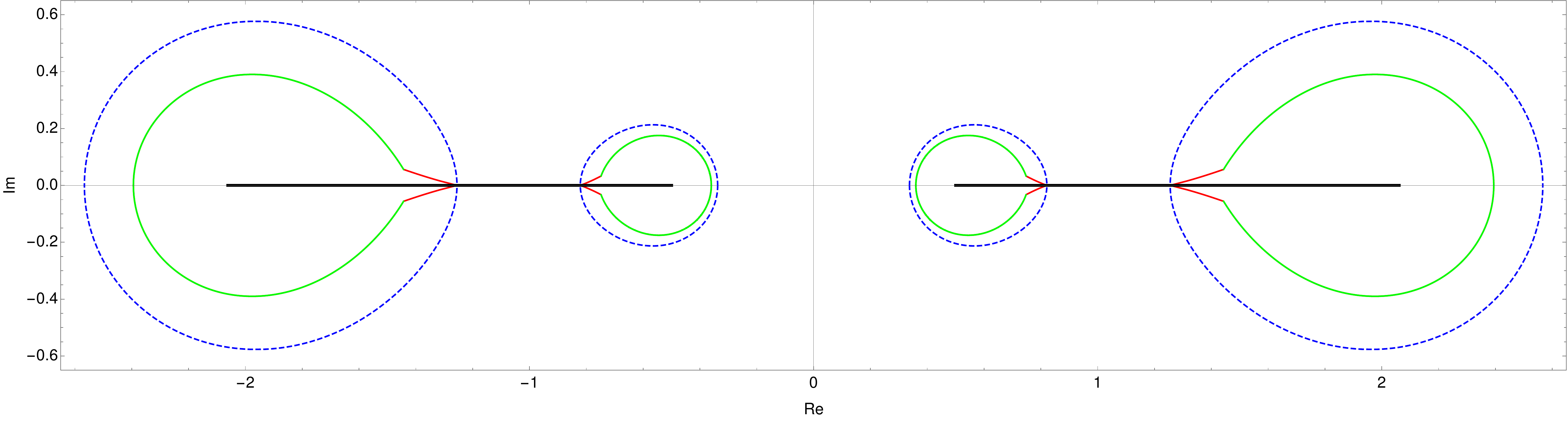}
	\vskip-6pt
    \caption{For $\|V\|_{1}=0.75$, all points of the boundary of the improved spectral enclosure except the red spikes are eigenvalues of particular discrete Dirac operators.}
    \end{subfigure}
    \begin{subfigure}[c]{0.9\textwidth}
    \vskip1pt
	\includegraphics[width=\textwidth]{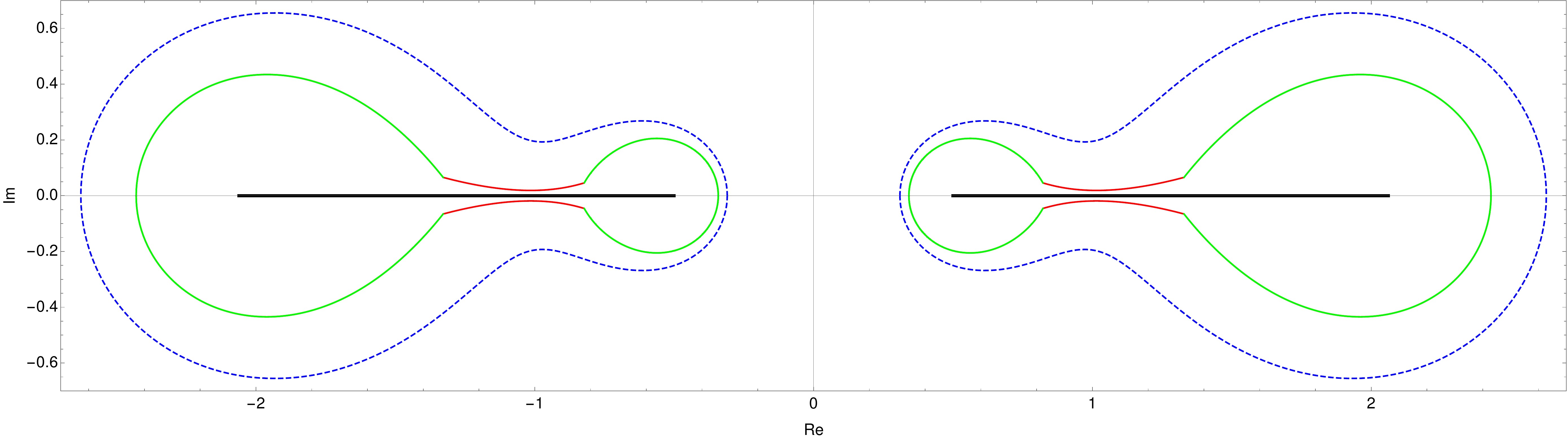}
	\vskip-6pt
	\caption{For $\|V\|_{1}=0.8$, all points of the boundary of the improved spectral enclosure except the two thin corridors connecting the rounded regions are eigenvalues of particular discrete Dirac operators.}
    \end{subfigure}
    \begin{subfigure}[c]{0.9\textwidth}
    \vskip1pt
	\includegraphics[width=\textwidth]{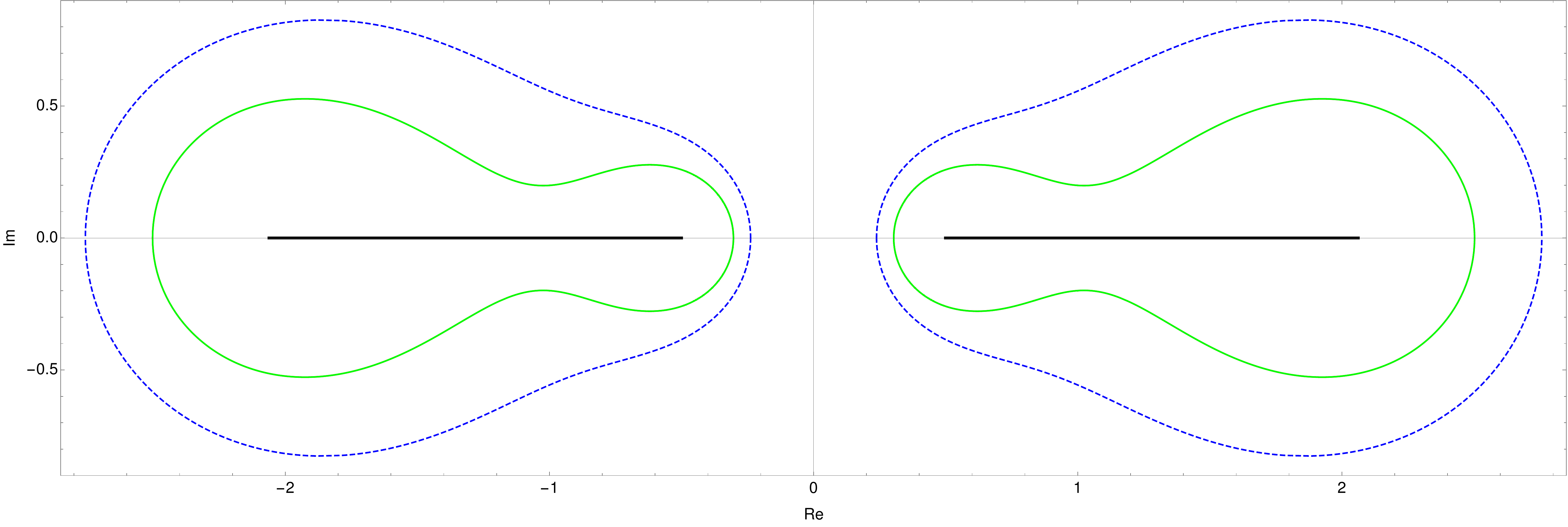}
	\vskip-6pt
	\caption{For $\|V\|_{1}=0.9$ (and higher values as well), it seems that all points of the boundary of the improved spectral enclosure are eigenvalues of particular discrete Dirac operators.}
    \end{subfigure}
     \caption{Boundary curves for the spectral enclosures Theorem~\ref{thm:1-norm} (dashed blue lines) compared to the corresponding result of Theorem~\ref{thm:spec_bound_improved} (solid red/green lines) with $m=1/2$. Green color demonstrates the partial optimality in the sense of Theorem~\ref{thm:partial_optim}.}
    \label{fig:sp_ell1_compare}
\end{figure}

\subsection*{C: Comparison plots for the $\ell^{p}$-bounds from Theorems~\ref{thm:p-norm.stein},\ref{thm:spec_bound_improved}, and Corollary~\ref{cor:p-norm}}

In the next plots, we compare the spectral enclosures given in Theorems~\ref{thm:p-norm.stein} and~\ref{thm:spec_bound_improved} for $\ell^{p}$-potentials with $p>1$. As an extra, we add also the spectral enclosure of Corollary~\ref{cor:p-norm} into these plots. In this numerical comparison, we exclude the result of Theorem~\ref{thm:stein-impr} due to its complexity and non-reliability of the numerical computations. Note that it is clear from the proofs that Theorem~\ref{thm:stein-impr} is an improvement of Theorem~\ref{thm:p-norm.stein}.

The comparison is made in plots in Figure~\ref{fig:ellp_compare} where the boundary curve of the the spectral enclosure from Theorem~\ref{thm:p-norm.stein} is made in \emph{solid yellow lines}, from Theorem~\ref{thm:spec_bound_improved} in \emph{red dashed lines}, and from Corollary~\ref{cor:p-norm} in \emph{blue dotted lines} for $m=1$, $\|V\|_{p}=0.7$, and four choices of the parameter $p\in\{3/2,2,3,5\}$.

It is by no means evident whether one of the spectral enclosures of Theorems~\ref{thm:p-norm.stein} and~\ref{thm:spec_bound_improved} is better than the other. However, numerical experiments indicate that none is better than the other, \ie$\,$~none is a subset of the other, in general.
\begin{figure}[H]
	\centering
	\begin{subfigure}[c]{0.49\textwidth}
		\includegraphics[width=\textwidth]{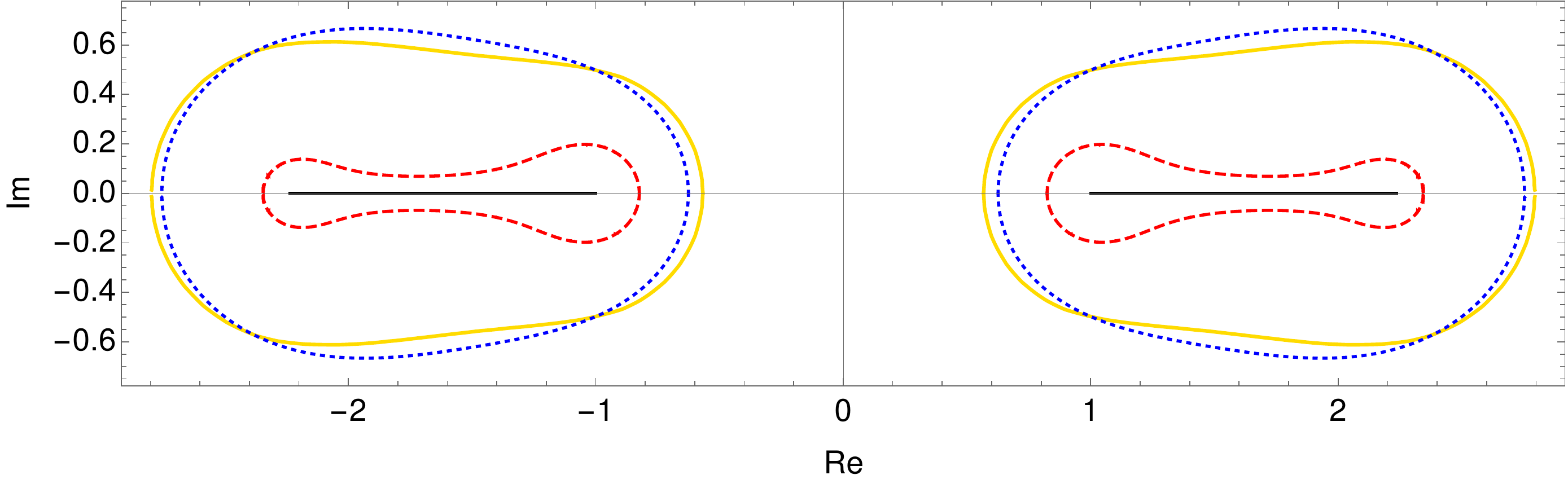}
		\caption{$p=3/2$}
	\end{subfigure}
	\begin{subfigure}[c]{0.49\textwidth}
		\includegraphics[width=\textwidth]{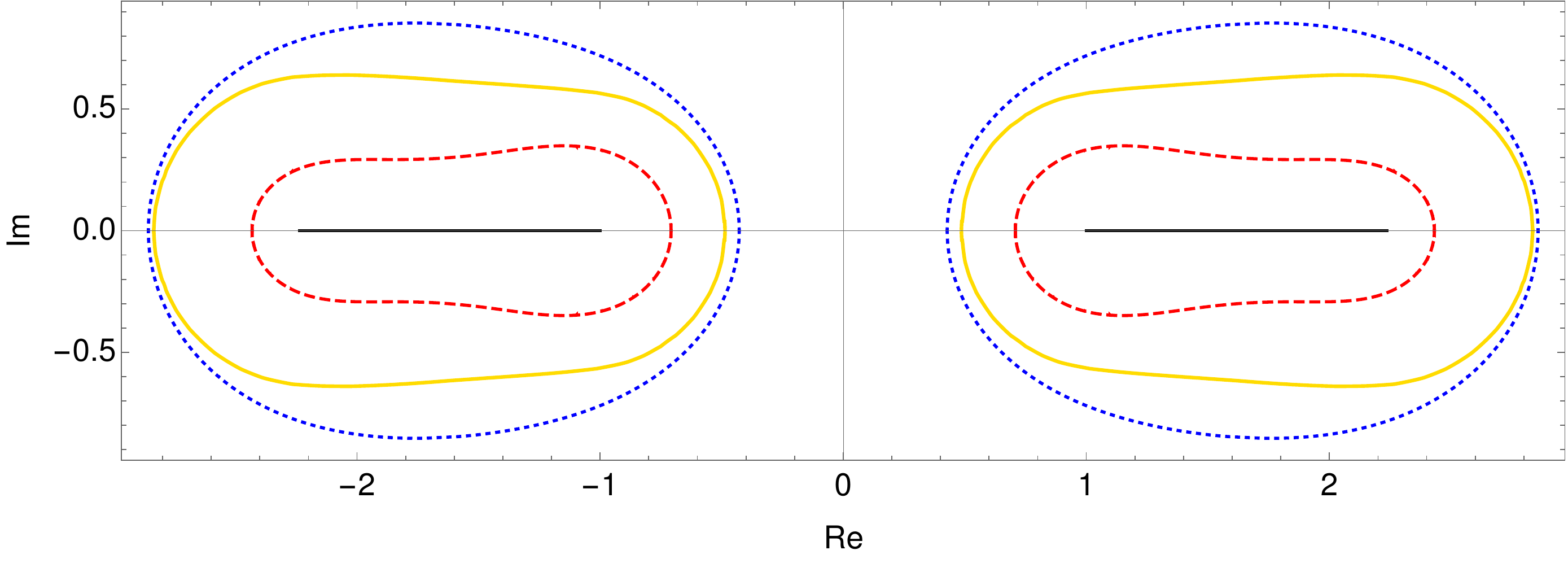}
		\caption{$p=2$}
	\end{subfigure}
	\vskip6pt
	\begin{subfigure}[c]{0.49\textwidth}
		\includegraphics[width=\textwidth]{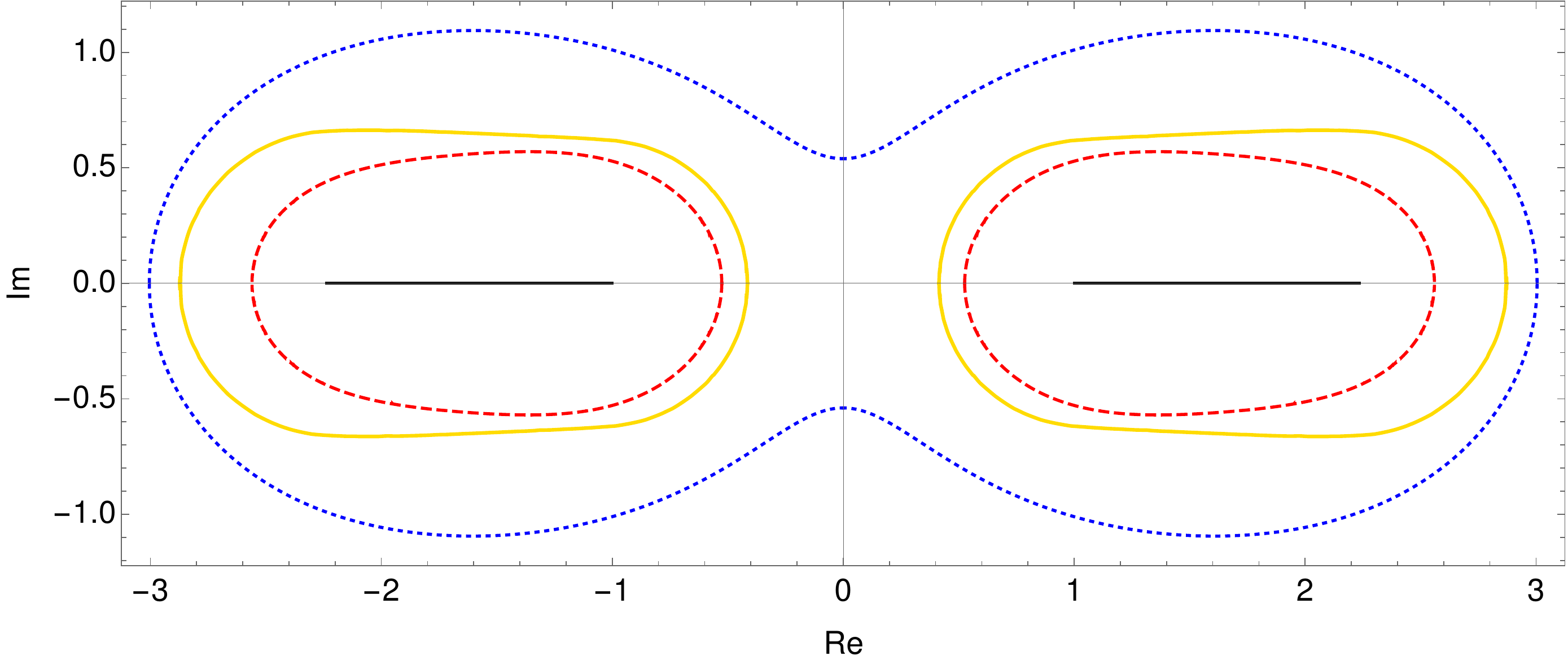}
		\caption{$p=3$}
	\end{subfigure}
	\begin{subfigure}[c]{0.49\textwidth}
		\includegraphics[width=\textwidth]{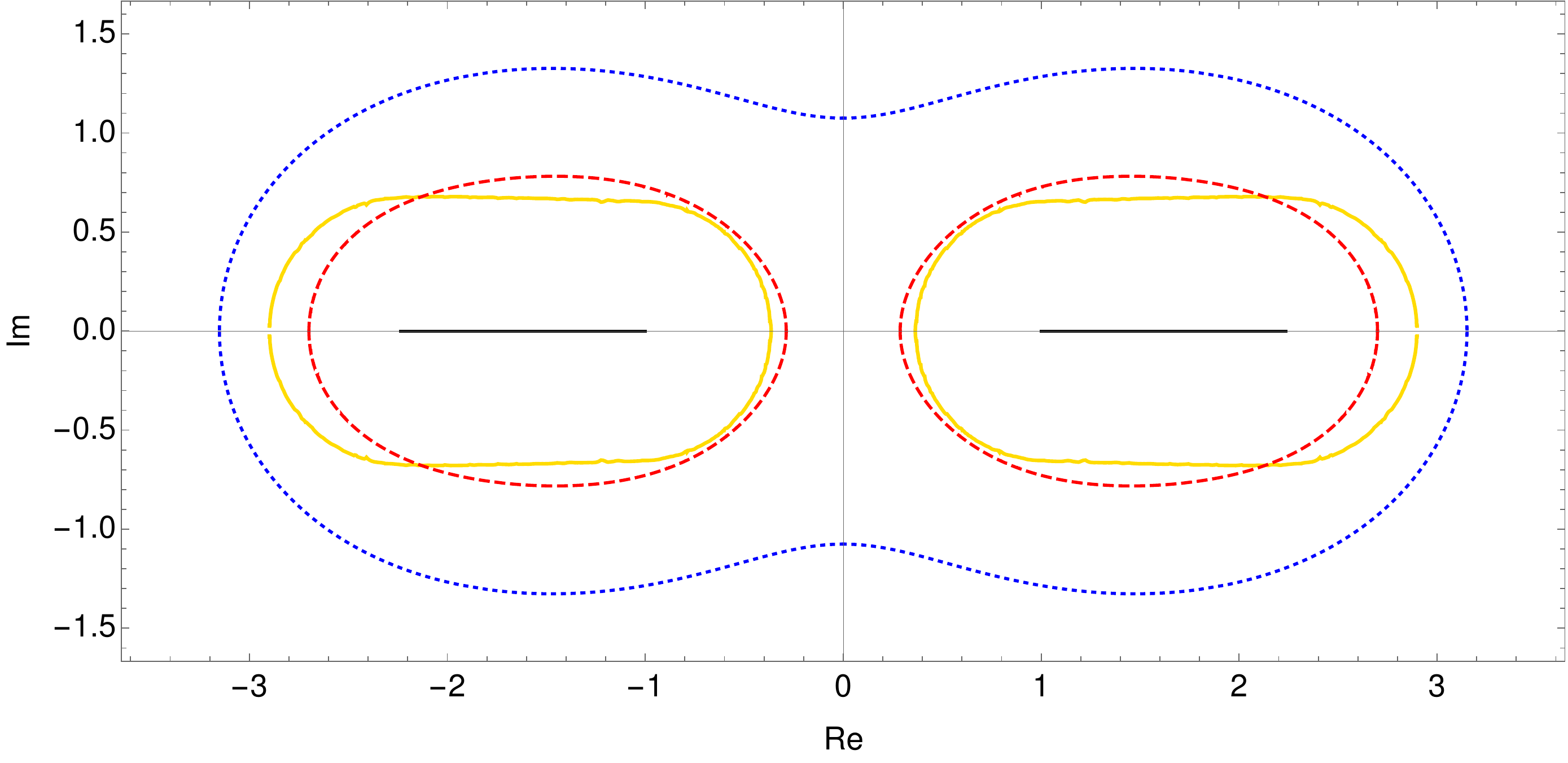}
		\caption{$p=5$}
	\end{subfigure}
	\caption{Boundary curves of spectral enclosures of Theorem~\ref{thm:p-norm.stein} (solid yellow lines), Theorem~\ref{thm:spec_bound_improved} (red dashed lines), and Corollary~\ref{cor:p-norm} (blue dotted lines) for $m=1$, $\|V\|_{p}=0.7$, and four choices of the parameter~$p>1$.}
	\label{fig:ellp_compare}
\end{figure}

\subsection*{D: A plot for Remark~\ref{rem:diagonal.dominance}}

Finally, as an illustration for Remark~\ref{rem:diagonal.dominance}, Figure~\ref{fig:T01ineq} shows for what $k$'s within the unit disk the norm of the diagonal element $T_{0}(k)$ of the resolvent~\eqref{eq:res_D0} is not dominant.
\begin{figure}[H]
	\centering
	\includegraphics[width=0.45\textwidth]{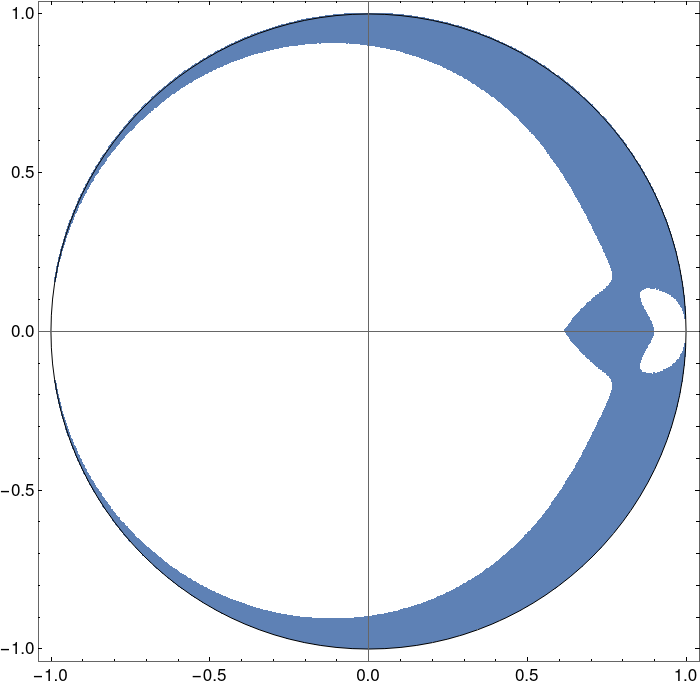}
	\caption{The blue subregion of the unit disk indicates the set of $k$'s for which $|T_0(k)|<|T_1(k)|$ when $m=1/8$.}
	\label{fig:T01ineq}
\end{figure}
%
%
%

\subsection*{Acknowledgment}
%
The research of B.C.~was partially supported by the grant 
No.~17-01706S of the Czech Science Foundation (GA\v{C}R) 
and by Fondo Sociale Europeo -- Programma Operativo 
Nazionale Ricerca e Innovazione 2014-2020, progetto 
PON:~progetto AIM1892920-attivit\`a 2, 
linea 2.1 -- CUP H95G18000150006 ATT2. 
The research of D.~K.~was partially supported 
by the GACR grant No.~18-08835S. 
F.{\v S}.~acknowledges financial support by the Ministry 
of Education, Youth and Sports of the Czech Republic 
project no.~CZ.02.1.01/0.0/0.0/16\_019/0000778
%


 \bibliography{bib_disc_dir_v20}
 \bibliographystyle{amsplain}

\end{document}